\documentclass[10pt,reqno]{amsart}

\usepackage{color}
\usepackage{mathrsfs}
\usepackage{enumitem}
\usepackage{graphicx}
\usepackage[normalem]{ulem}
\usepackage{amsmath, amssymb, amsfonts, dsfont, amsthm}
\usepackage[utf8]{inputenc}
\DeclareGraphicsExtensions{.pdf,.png,.jpg}
\newtheorem{thm}{Theorem}[section]
\newtheorem{defn}[thm]{Definition}
\newtheorem{corollary}[thm]{Corollary}
\newtheorem{lemma}[thm]{Lemma}

\theoremstyle{remark}
\newtheorem{remark}[thm]{Remark}

\usepackage{comment}
\def\XXint#1#2#3{{\setbox0=\hbox{$#1{#2#3}{\int}$}
\vcenter{\hbox{$#2#3$}}\kern-.5\wd0}}

\newcommand\cbrk{\text{$]$\kern-.15em$]$}}
\newcommand\opar{\text{\,\raise.2ex\hbox{${\scriptstyle
|}$}\kern-.34em$($}}
\newcommand\cpar{\text{$)$\kern-.34em\raise.2ex\hbox{${\scriptstyle |}$}}\,}

\def\<{\langle}
\def\>{\rangle}

\def\E{{\mathbb E}}

\newcommand\bL{\mathbb{L}}
\newcommand\bR{\mathbb{R}}

\newcommand\bH{\mathbb{H}}
\newcommand\bZ{\mathbb{Z}}

\newcommand\bD{\mathbb{D}}
\newcommand\bS{\mathbb{S}}

\newcommand\bN{\mathbb{N}}

\newcommand\cA{\mathcal{A}}
\newcommand\cB{\mathcal{B}}

\newcommand\cD{\mathcal{D}}
\newcommand\cF{\mathcal{F}}

\newcommand\cK{\mathcal{K}}

\newcommand\cP{\mathcal{P}}

\newcommand\cT{\mathcal{T}}
\newcommand\cO{\mathcal{O}}

\def\R {{\mathbb R}}

\newcommand\frH{\mathfrak{H}}

\newcommand{\mysection}[1]{\section{#1}
\setcounter{equation}{0}}

\newcommand{\one}{\ensuremath{\mathds 1}}

\newcommand{\prob}{\ensuremath{\mathbb P}}

\newcommand{\pred}{\ensuremath{\mathcal{P}}}

\newcommand{\cont}{\ensuremath{\mathcal{C}}}

\newcommand{\abs}[1]{\ensuremath{\lvert #1 \rvert}}
\newcommand{\Abs}[1]{\ensuremath{\big\lvert  #1 \big\rvert}}

\newcommand{\nnrm}[2]{\ensuremath{\lVert #1 \rVert_{#2}}}
\newcommand{\gnnrm}[2]{\ensuremath{\big\lVert #1 \big\rVert_{#2}}}
\newcommand{\sgnnrm}[2]{\ensuremath{\Big\lVert #1 \Big\rVert_{#2}}}

\newcommand{\nrklam}[1]{(#1)}

\newcommand{\grklam}[1]{\big(#1\big)}
\newcommand{\sgrklam}[1]{\Big(#1\Big)}
\newcommand{\ssgrklam}[1]{\bigg(#1\bigg)}

\newcommand{\ggklam}[1]{\big\{#1\big\}}

\newcommand{\ssggklam}[1]{\bigg\{#1\bigg\}}

\newcommand{\sgeklam}[1]{\Big [#1 \Big ]}

\newcommand{\domain}{\ensuremath{\mathcal{O}}}   % beschränktes Lipschitz-Gebiet
\newcommand{\dist}{\ensuremath{\rho}}   % Distanz zum Rand
\DeclareFontFamily{U}{matha}{\hyphenchar\font45}
\DeclareFontShape{U}{matha}{m}{n}{
      <5> <6> <7> <8> <9> <10> gen * matha
      <10.95> matha10 <12> <14.4> <17.28> <20.74> <24.88> matha12
      }{}
\DeclareSymbolFont{matha}{U}{matha}{m}{n}
\DeclareFontSubstitution{U}{matha}{m}{n}
\DeclareFontFamily{U}{mathx}{\hyphenchar\font45}
\DeclareFontShape{U}{mathx}{m}{n}{
      <5> <6> <7> <8> <9> <10>
      <10.95> <12> <14.4> <17.28> <20.74> <24.88>
      mathx10
      }{}
\DeclareSymbolFont{mathx}{U}{mathx}{m}{n}
\DeclareFontSubstitution{U}{mathx}{m}{n}
\DeclareMathDelimiter{\vvvert}{0}{matha}{"7E}{mathx}{"17}
\newcommand{\wso}{\ensuremath{K}}
\newcommand{\bwso}{\ensuremath{\mathbb{\wso}}}
\excludecomment{suggestion}

\excludecomment{pcidetail}

\definecolor{felix}{rgb}{0.2,0.2,1.0} % eine Art von Blau fuer Felix
\definecolor{petru}{rgb}{0.7,0.1,0.1} % ein dunkles Rot Petru
\definecolor{alternative}{rgb}{0.1,0.1,0.7} % 
\definecolor{detail}{rgb}{0.0,0.5,0.0} % a dark green for Details

\newcommand{\icp}{\color{black}}

\begin{document}

\title[The stochastic heat equation on polygonal domains]
{On the regularity  of  the stochastic heat equation
 on polygonal domains in $\bR^2$}

\author{Petru A. Cioica-Licht}
\thanks{The first named author has been partially supported by the Marsden Fund Council from Government funding, administered by the Royal Society of New Zealand, and by a University of Otago Research Grant (114023.01.R.FO). The  research of the second and third author was supported by Basic Science Research Program through the National Research Foundation of Korea (NRF) 
funded by the Ministry of Education (NRF-2017R1D1A1B03033255) and (NRF-2013R1A1A2060996), respectively.
The authors would like to thank Felix Lindner for his contribution at an early stage of this manuscript}
\address{Petru A. Cioica-Licht (n\'e Cioica), Department of Mathematics and Statistics, University of Otago, PO Box~56, Dunedin 9054, New Zealand}
\email{pcioica@maths.otago.ac.nz}
\author{Kyeong-Hun Kim}
%\thanks{The  research of the second author was supported by Basic Science Research Program through the National Research Foundation of Korea (NRF) funded by the Ministry of Education  (NRF-2017R1D1A1B03033255)}
\address{Kyeong-Hun Kim, Department of Mathematics, Korea University, Anam-ro 145, Sungbuk-gu, Seoul, 02841, Republic of Korea}
\email{kyeonghun@korea.ac.kr}
\author{Kijung Lee}
%\thanks{The research of the third author was supported by Basic Science Research Program through the National Research Foundation of Korea (NRF) funded by the Ministry of Education, Science and Technology (NRF-2013R1A1A2060996)}
\address{Kijung Lee, Department of Mathematics, Ajou University, Worldcup-ro 206, Yeongtong-gu, Suwon, 16499, Republic of Korea}
\email{kijung@ajou.ac.kr}



\begin{abstract}
We establish existence, uniqueness and higher order weighted $L_p$-Sobolev regularity for the stochastic heat equation with zero Dirichlet boundary condition on angular domains and on  polygonal domains in $\mathbb{R}^2$.
We use a system of mixed weights consisting of appropriate powers of the distance to the vertexes and of the distance to the boundary to measure the regularity with respect to the space variable.
In this way we can capture the influence of both main sources for singularities: the incompatibility between noise and boundary condition on the one hand and the singularities of the boundary on the other hand.
The range of admissible powers of the distance to the vertexes is described in terms of the maximal interior angle and is sharp. 
\end{abstract}

\maketitle

\mysection{Introduction}\label{sec:Introduction}

In this article we continue the analysis started in~\cite{CioKimLee+2018} towards a refined $L_p$-theory for second order stochastic partial differential equations (SPDEs, for short) on non-smooth domains.
The main challenges in the construction of such a theory come from two effects that are known to influence the regularity of the solution:
On the one hand, the incompatibility between noise and boundary condition results in blow-ups of the higher order derivatives near the boundary---even if the boundary is smooth. 
On the other hand, the singularities of the boundary cause a similar effect in their vicinity---even if the forcing terms are deterministic.
We refer to the introduction of~\cite{CioKimLee+2018} and the literature therein for details. 

The well developed $L_p$-theory for second order SPDEs on smooth domains, 
carried out within the analytic approach initiated by N.V.~Krylov,
shows that the incompatibility between noise and boundary condition can be captured very accurately by using a system of weights based on the distance to the boundary, see, for instance,~\cite{Kim2004, KimKry2004, Kry1994,KryLot1999, KryLot1999b}.
Moreover, the results in~\cite{CioKimLee+2018} indicate that a system of weights based on the distance to a corner of the underlying domain is suitable to describe the impact of this boundary singularity on the solution.
Thus, in order to capture both effects, a system based on a combination of appropriate powers of the distance to the boundary and of the distance to the boundary singularities suggests itself.

Our primary goal in this article is to show how such a system of mixed weights can be used in order to provide higher order spatial weighted $L_p$-Sobolev regularity for second order SPDEs with zero Dirichlet boundary condition on angular domains and on polygonal domains in $\bR^2$. 
For the moment we restrict ourselves to the stochastic heat equation, since already the analysis of this equation involves many non-trivial steps and has been a persisting problem for a long time.
At the same time, we believe that in this way we can shade some light on the general strategy without getting lost in details.
Our general setting is as follows: 
Let $(w^k_t)$, $k\in\bN$, be a sequence of independent real-valued standard Brownian motions on a probability space $(\Omega,\cF,\prob)$ and let  $T\in (0,\infty)$ be a finite time horizon.
We consider the stochastic heat equation 
\begin{equation}\label{eq:SHE:Intro:a}
\left.
\begin{alignedat}{3}
du 
&= 
\grklam{ \Delta && u + f^0+f^i_{x_i}}\,dt
+
g^k\,dw^k_t \quad \text{on } \Omega\times(0,T]\times\domain,	\\
u
&=
0 && \quad \text{on } \Omega\times(0,T]\times\partial\domain,	\\
u(0)
&=
0 && \quad \text{on } \Omega\times\domain,
\end{alignedat}
\right\}	
\end{equation}
on various types of domains $\domain\subseteq\bR^d$.
Our focus lies in particular on polygonal domains and on angular domains $\domain \subseteq\bR^2$. 
Note that, as usual, here and in the sequel we use the Einstein summation convention on the repeated indexes $i$ and $k$.

Our main results address the existence, uniqueness and higher order spatial regularity of the solution to Equation~\eqref{eq:SHE:Intro:a} on angular domains and on polygonal domains $\domain\subset\bR^2$. 
By using a weight system based solely on the distance to the set of vertexes of $\domain$, we establish existence and uniqueness of a solution to Equation~\eqref{eq:SHE:Intro:a} with suitable weighted $L_p$-Sobolev regularity of order one with respect to the space variable; see Theorem~\ref{thm:ex:uni:2DCone} (angular domains) and Theorem~\ref{thm polygon main} (polygonal domains). 
The lower bound of the range~\eqref{eq:range:vertex} for the weight parameter $\theta$, which corresponds to the best integrability property of the solution near the vertex, is sharp; see also the introduction of~\cite{CioKimLee+2018}. 
Moreover, by using, in addition, appropriate powers of the distance to the boundary $\partial \domain$ we describe the behavior of higher order spatial derivatives of the solution; see Corollary~\ref{high} (angular domains) and Theorem~\ref{thm_polygons_1} (polygonal domains).  

The key estimate, which paves the way for all the results mentioned above, is presented in Theorem~\ref{lem:estim:2DCone:Theta}. 
Roughly speaking, it shows which system of weights is suitable in order to be able to lift the spatial regularity of the solution of the stochastic heat equation~\eqref{eq:SHE:Intro:a} on an angular domain with the regularity of the forcing terms. In short, it can be stated as follows:
Let 
\begin{equation}\label{domain:angular}
\cD:=\cD_{\kappa_0}:=\ggklam{x\in \bR^2: x=(r\cos\vartheta,r\sin\vartheta),\; r>0,\;\vartheta\in (0,\kappa_0)},
\end{equation}
be an angular domain with vertex at the origin and angle~$\kappa_0\in(0,2\pi)$.
Moreover, let $\dist(x):=\dist_\cD(x):=\mathrm{dist}(x,\partial\cD)$ be the distance of a point $x\in\cD$ to the boundary $\partial\cD$ of $\cD$. If $u$ is the solution to Equation~\eqref{eq:SHE:Intro:a} on $\cD$,
then, for arbitrary $m\in \bN$, $1<\Theta<p+1$ and $\theta\in\bR$,
% fulfilling~\eqref{eq:range:vertex:Intro}, 
we can estimate
\[
\E\int_0^T \ssgrklam{\sum_{\abs{\alpha}\leq m}\int_\cD \Abs{ \rho(x)^{\abs{\alpha}-1} D^\alpha u(t,x)}^p\abs{x}^{\theta-2}\ssgrklam{\frac{\rho(x)}{\abs{x}}}^{\Theta-2}\,dx}\,dt
\]
by the weighted $L_p$-norm 
\[
\E\int_0^T\int_\cD \Abs{ \abs{x}^{-1} u(t,x)}^p \abs{x}^{\theta-2}\ssgrklam{\frac{\rho(x)}{\abs{x}}}^{\Theta-2}\,dx\,dt
\]
of the solution plus appropriate weighted $L_p$-Sobolev norms of the forcing terms $f^0$, $f^i$ and $g$, of order $(m-2)\lor 0$, $m-1$, and $m-1$, respectively.
It is worth mentioning that the range for the parameter $\Theta$ in this estimate is natural and sharp, see Remark~\ref{explanation of key} for details.

As already mentioned above, there already exists a comprehensive $L_p$-regularity theory for second order SPDEs in weighted Sobolev spaces with weights based solely on the distance to the boundary, see also~\cite{Kim2014} in addition to the reference given above. Typically, the solution $u$ to Equation~\eqref{eq:SHE:Intro:a} on a domain $\domain\subseteq\bR^d$ fulfills
\[
\E\int_0^T \ssgrklam{\sum_{\abs{\alpha}\leq m}\int_\domain \Abs{\rho^{\abs{\alpha}-1}D^\alpha u}^p\rho^{\Theta-d}\,dx}\,dt<\infty,
\]
provided that the domain $\cO$ is sufficiently smooth and the free terms $f^0$, $f^i$ and $g$ are in corresponding weighted Sobolev spaces. More precisely,
on smooth domains, i.e., at least $\cont^1$, such a theory is possible with
\[
d-1<\Theta<d+p-1,
\]
see, e.g.,~\cite[Remark~2.7]{Kim2004}.
However, on non-smooth domains, only $\Theta\in (d+p-2-\varepsilon,d+p-2+\varepsilon)$ is possible with a small $\varepsilon>0$ that depends on the roughness of the boundary of the domain and is not explicitly given~\cite{Kim2014}. In particular, for large $p>2$, $\Theta=d$ is not admissible, see~\cite[Example~2.17]{Kim2014} for a typical counterexample.
Our results show that, on polygonal domains, if we use an appropriate power of the distance to the set of vertexes to control the behavior of the solution in their proximate vicinity, then $\Theta=d=2$ is possible away from the vertexes.

Our analysis takes place within the framework of the analytic approach.
The proofs of the main results rely on a mixture of Green function estimates on angular domains, suitable localization techniques and some delicate estimates for the stochastic heat equation on $\cont^1$ domains.
Alternatively, one could think of Equation~\eqref{eq:SHE:Intro:a} as an abstract Banach space valued stochastic evolution equation and try to obtain a similar theory by using the extension of the semigroup approach for SPDEs to Banach spaces developed by J.M.A.M.~van Neerven, M.C.~Veraar and L.~Weis~\cite{NeeVerWei2008, NeeVerWei2012, NeeVerWei2012b}. However, for this to succeed, one would have to (at least!) check whether the (properly defined) Dirichlet Laplacian on weighted Sobolev spaces has an appropriate functional calculus. 
Moreover, one would need a description of the domain of the square root of this operator in terms of suitable weighted Sobolev spaces.
To the best of our knowledge, both questions are not trivial and yet to be answered.
In this context it is worth mentioning that the recently developed Calder\'on-Zygmund theory for singular stochastic integrals from~\cite{LorVer2019+} together with the $L_p$-theory developed in~\cite{CioKimLee+2018} lead to an $L_q(L_p)$-theory with $q\neq p$ without making use of precise descriptions of the domains of fractional powers of the Laplacian nor of the existence of a bounded $H^\infty$-calculus, see~\cite[Example~8.12]{LorVer2019+}.

This article is organized as follows: In Section~\ref{sec:2DCone} we present and prove the main results concerning existence, uniqueness (Theorem~\ref{thm:ex:uni:2DCone}) and higher order regularity (Corollary~\ref{high}) of the stochastic heat equation on angular domains.
The proofs rely on two key estimates, which are stated in Theorem~\ref{lem:estim:2DCone:Theta} and Lemma~\ref{lem 4.5.1} and proven in detail in Section~\ref{sec:proof:lift} and Section~\ref{4}, respectively. 
Finally, in Section~\ref{sec:Polygons} we present our analysis of the stochastic heat equation on polygonal domains.
Before we start, we fix some notation.

\medskip

\noindent\textbf{Notation.}  
Throughout this article, $(\Omega,\cF,\prob)$ is a complete probability space and $\nrklam{\cF_{t}}_{t\geq0}$ is an increasing filtration of $\sigma$-fields $\cF_{t}\subset\cF$, each of which contains all $(\cF,\prob)$-null sets.
We assume that on $\Omega$ we are given a family $(w_t^k)_{t\geq0}$, $k\in\bN$, of independent one-dimensional Wiener processes relative to $\nrklam{\cF_{t}}_{t\geq0}$. By $\cP$ we denote the predictable $\sigma$-algebra on $\Omega\times (0,\infty)$ generated by $\nrklam{\cF_{t}}_{t\geq0}$ and any of its trace $\sigma$-algebras.
Moreover, $T\in(0,\infty)$ is a finite time horizon and $\Omega_T:=\Omega\times (0,T]$. 
For a measure space $(A, \cA, \mu)$, a Banach space $B$ and $p\in[1,\infty)$, we write $L_p(A,\cA, \mu;B)$ for the collection of all $B$-valued $\bar{\cA}$-measurable functions $f$ such that 
$$
\|f\|^p_{L_p(A,\cA,\mu;B)}:=\int_{A} \lVert f\rVert^p_{B} \,d\mu<\infty.
$$
Here $\bar{\cA}$ is the completion of $\cA$ with respect to $\mu$.  The Borel $\sigma$-algebra on a topological space $E$ is denoted by $\cB(E)$. We will drop $\cA$ or $\mu$ in $L_p(A,\cA, \mu;B)$  when the $\sigma$-algebra $\cA$ or the measure $\mu$  are obvious from the context. 
For functions $f$ depending on $\omega\in \Omega$, $t\geq 0$ and $x\in\bR^d$, we usually drop the argument $\omega$, and denote them by $f(t,x)$. 
If $\domain\subseteq\bR^d$ is a domain in $\bR^d$, we write $\cont^{\infty}_c(\domain)$ for the space of infinitely differentiable functions with compact support in $\domain$. Moreover, $\cont^{2}_c(\domain)$ is the space of twice continuously differentiable functions with compact support in $\domain$.
For a function $f\colon\domain\to\bR$ and any multi-index $\alpha=(\alpha_1,\ldots,\alpha_d)$, $\alpha_i\in \{0,1,2,\ldots\}$,   
$$
D^{\alpha}f(x):=\partial^{\alpha_d}_d\cdots\partial^{\alpha_1}_1u(x),
\quad x=(x^1,\ldots,x^d),
$$
where $\partial^{\alpha_i}_i=\frac{\partial^{\alpha_i}}{\partial (x^i)^{\alpha_i}}$ is the $\alpha_i$ times (generalized) derivative with respect to the $i$-th coordinate;
%We mention that the order of differentiation does not matter for weak derivative $D^{\alpha}u$.
$f_{x^i}:=\frac{\partial}{\partial x^i}u$.
By making slight abuse of notation, for $m\in\{0,1,2,\ldots\}$, we write $D^m f$ for any (generalized) $m$-th order derivative of $f$ and for the vector of all $m$-th order derivatives. For instance, if we write $D^mf\in B$, where $B$ is a function space on $\domain$, we mean $D^\alpha\in B$ for all multi-indexes $\alpha$ with $\abs{\alpha}=m$.
The notation $f_x$ is used synonymously for $D^1f$, whereas $\nnrm{f_x}{B}:=\sum_i\nnrm{f_{x^i}}{B}$.
Throughout the article, the letter $N$ is used to denote a finite positive constant that may differ from one appearance to another, even in the same chain of inequalities.
When we write $N=N(a,b,\cdots)$, we mean that $N$ depends only on the parameters inside the parentheses.
 Moreover, $A\sim B$ is short for `$A\leq N B$ and $B\leq N A$'.

\mysection{The  stochastic heat equation on  angular domains }\label{sec:2DCone}

In this section we present our analysis for the stochastic heat equation 
\begin{eqnarray}\label{eq:SHE:Intro}
du=(\Delta u+f^0+ f^i_{x^i})\,dt+ g^k \,dw^k_t, \quad t\in (0,T], %\\
%&& u(t,\cdot)=0, \quad u|_{\partial \cD}=0  \nonumber
\end{eqnarray}
on angular domains $\cD\subseteq\bR^2$ with zero Dirichlet boundary condition and vanishing initial value.
We establish existence and uniqueness (Theorem~\ref{thm:ex:uni:2DCone}) as well as higher order spatial regularity of the solution (Corollary~\ref{high}) within a framework of weighted Sobolev spaces. 
The weights are products of appropriate powers of the distance to the vertex and of the distance to the boundary (two infinite edges and the vertex). 
The key estimate, which enables us to describe the behavior of the higher order derivatives of $u$ near the boundary even if the forcing terms behave badly near the boundary but are sufficiently smooth inside the domain, is presented in Theorem~\ref{lem:estim:2DCone:Theta}, see also Remark~\ref{explanation of key}.

 To state our results, we first introduce appropriate function spaces.
The notation is mainly borrowed from~\cite{CioKimLee+2018}. Throughout, $\cD=\cD_{\kappa_0}$ is as defined in~\eqref{domain:angular} with $\kappa_0\in(0,2\pi)$ and $\rho_\circ(x):=\abs{x}$ denotes the distance of a point $x\in \cD$ to the origin (the only vertex of $\cD$).
Let $p>1$ and $\theta\in\bR$.
We write
$$
L^{[\circ]}_{p,\theta}(\cD):=L_{p}(\cD,\cB(\cD),\rho_\circ^{\theta-2} dx;\R)
\quad\text{and}\quad
L^{[\circ]}_{p,\theta}(\cD; \ell_2):=L_p(\cD,\cB(\cD),\rho_\circ^{\theta-2} dx;\ell_2)
$$
for the weighted $L_p$-spaces of real-valued and $\ell_2$-valued functions with weight $\rho_\circ^{\theta-2}$. For $n\in\{0,1,2,\ldots\}$ let
\[
\wso^n_{p,\theta}(\cD)
=
\ssggklam{
f : \nnrm{f}{\wso^n_{p,\theta}(\cD)} := \ssgrklam{\sum_{\abs{\alpha}\leq n} 
\gnnrm{\rho_{\circ}^{\abs{\alpha}} D^\alpha f}{L_{p,\theta}^{[\circ]}(\cD)}^p}^{1/p}
< 
\infty
},
\]
and define $K^n_{p,\theta}(\cD;\ell_2)$ accordingly.  Note that
\[
K^0_{p,\theta}(\cD)=L^{[\circ]}_{p,\theta}(\cD)
\quad\text{and}\quad
K^0_{p,\theta}(\cD;\ell_2)=L^{[\circ]}_{p,\theta}(\cD;\ell_2).
\]
Moreover, we write $\mathring{\wso}^1_{p,\theta}(\cD)$ for the closure in $\wso^1_{p,\theta}(\cD)$ of the space $\cont^\infty_c(\cD)$ of test functions.

%We will frequently use the following basic properties of the spaces $K^n_{p,\theta}(\cD)$. 
The weighted Sobolev spaces introduced above are classical examples of Kondratiev spaces. For their basic properties as well as their relevance in the analysis of elliptic partial differential equations on domains with conical singularities we refer to~\cite[Part~2]{KozMazRos1997}, see also the pioneering works \cite{Kon1967,Kon1970,KonOle1983, KufOpi1984}. 
In the sequel, we will frequently use the following basic properties.
They are mainly a consequence of the fact that 
for any multi-index $\alpha$
$$
\sup_{\cD} \left(\rho^{|\alpha|-1}_{\circ} |D^{\alpha}\rho_{\circ}|\right)\le N(\alpha)<\infty;
$$
the proof is left to the reader. 
\begin{lemma}
\label{lem 1}
Let $p>1, \theta\in \bR$ and  $n\geq 1$.  If $\alpha$ is a multi-index with 
$|\alpha|\leq n$, then
$$
\|\rho_{\circ}^{|\alpha|} D^{\alpha} f\|_{K^{n-|\alpha|}_{p,\theta}(\cD)}+\|D^{\alpha}  (\rho^{|\alpha|}_{\circ} f)\|_{K^{n-|\alpha|}_{p,\theta}(\cD)} \leq  N \|f\|_{K^{n}_{p,\theta}(\cD)},
$$
and
$$
\|f_{x}\|_{K^{n-1}_{p,\theta}(\cD)}\leq  N\|f\|_{K^{n}_{p,\theta-p}(\cD)},
$$
with $N$ independent of $f$.
\end{lemma}

 To formulate our conditions on the different parts of the equations, we will use the $L_p$-spaces of predictable stochastic processes on $\Omega_T:=\Omega\times(0,T]$ taking values in the weighted Sobolev spaces introduced above. 
For $p>1$, $\theta\in\R$, and $n\in\{0,1,2,\ldots\}$, we abbreviate 
$$
\bwso^{n}_{p,\theta}(\cD,T)
:=
L_p(\Omega_T, \pred,\prob\otimes dx;\wso^{n}_{p,\theta}(\cD)),
$$
$$
 \bwso^{n}_{p,\theta}(\cD,T;\ell_2)
:=
L_p(\Omega_T, \pred,\prob\otimes dx;\wso^{n}_{p,\theta}(\cD;\ell_2)),
$$
$$
\bL^{[\circ]}_{p,\theta}(\cD,T):=\bwso^0_{p,\theta}(\cD,T), \quad  \bL^{[\circ]}_{p,\theta}(\cD,T;\ell_2)
:=\bwso^0_{p,\theta}(\cD,T;\ell_2),
$$
and
\[
\mathring{\bwso}^1_{p,\theta}(\cD,T)
:=
L_p(\Omega_T, \pred,\prob\otimes dx;\mathring{\wso}^{1}_{p,\theta}(\cD)).
\]

Using these spaces we introduce the following classes of stochastic processes that are tailor-made for the analysis of Equation~\eqref{eq:SHE:Intro} on $\cD$.

\begin{defn}
   \label{defn sol}
 For $p\geq 2$  and $\theta\in \bR$
we write $u\in\mathcal{K}^1_{p,\theta,0}(\cD,T)$  if
$u\in\mathring{\bwso}^1_{p,\theta-p}(\cD,T)$
 and
there exist $f^0\in \bL^{[\circ]}_{p,\theta+p}(\cD,T), f^i\in  \bL^{[\circ]}_{p,\theta}(\cD,T)$, $i=1,2$, and  $g\in \bL^{[\circ]}_{p,\theta}(\cD,T;\ell_2)$,
such that 
\begin{equation}\label{eqn 28}
 du=(f^0+f^i_{x^i})\, dt +g^k \, dw^k_t,\quad t\in(0,T],
\end{equation}
 on $\cD$  in the sense of distributions with $u(0,\cdot)=0$, that is, for any $\varphi \in
\cont^{\infty}_{c}(\cD)$, with probability one, the equality
\begin{equation}\label{eq:distribution}
(u(t,\cdot),\varphi)=   \int^{t}_{0}
\left[(f^0(s,\cdot),\varphi) -(f^i(s,\cdot),\varphi_{x^i}) \right]ds + \sum^{\infty}_{k=1} \int^{t}_{0}
(g^k(s,\cdot),\varphi)\,  dw^k_s
\end{equation}
holds for all $t \leq T$. 
In this situation
we also write
$$
\bD u:=f^0+f^i_{x^i}\qquad\text{and}\qquad \bS u :=g
$$
for the deterministic part and the stochastic part, respectively.
\end{defn}

\begin{remark}
% For $p\geq 2$, 
The spaces $\cK^1_{p,\theta,0}(\cD,T)$ from Definition~\ref{defn sol} coincide with the spaces $\mathfrak{K}^1_{p,\theta}(\cD,T)$ introduced in~\cite[Definition~3.4]{CioKimLee+2018}.
The only (apparent) difference is that in the definition of $\mathfrak{K}^1_{p,\theta}(\cD,T)$ the deterministic part $\bD u$ is required to be an element of $\mathbb K^{-1}_{p,\theta+p}(\cD,T):=L_p(\Omega_T;K^{-1}_{p,\theta+p}(\cD))$, where $K^{-1}_{p,\theta+p}(\cD)$ is the dual of $\mathring{K}^1_{p',\theta'-p'}(\cD)$ with $1/p+1/p'=1$ and $\theta/p+\theta'/p'=2$.
However, this is not really a difference, since 
$$
\mathbb K^{-1}_{p,\theta+p}(\cD,T)=\ggklam{f^0+f^i_{x^i} : f^0\in \bL^{[\circ]}_{p,\theta+p}(\cD,T), f^i \in \bL^{[\circ]}_{p,\theta}(\cD,T)}.
$$
This can be proven with a similar strategy as the analogous result for classical Sobolev spaces, see, e.g., \cite[page~62ff.]{AdaFou2003}, and by using the fact that for reflexive Banach spaces $B$ with dual $B^*$, the dual of $L_{p'}(\Omega_T;B)$ is isometrically isomorphic to $L_p(\Omega_T;B^*)$, see, e.g., \cite[Theorem~IV.1.1 and Corollary~III.2.13]{DieUhl1977}.

\end{remark}

 In this article, Equation~\eqref{eq:SHE:Intro} has the following meaning on $\cD$.

\begin{defn}\label{def:solution:D}
We say that  $u$ is a solution to Equation~\eqref{eq:SHE:Intro} on $\cD$ in the class $\mathcal{K}^{1}_{p,\theta,0}(\cD,T)$ 
if
$u\in \mathcal{K}^{1}_{p,\theta,0}(\cD,T)$ with
\[
\bD u = \Delta u + f^0+f^i_{x^i}=f^0+(f^i+u_{x^i})_{x^i}
\qquad
\text{and}
\qquad
\bS u = g.
\]
\end{defn}

Now that we have specified the setting, we are ready to present our results. We start with the key estimate in this article. Its proof is given in Section~\ref{sec:proof:lift}.
Recall that $\rho(x):=\rho_{\cD}(x):=\mathrm{dist}(x,\partial\cD)$ denotes the distance of a point $x\in\cD$ to the boundary $\partial \cD$.

\begin{thm}\label{lem:estim:2DCone:Theta}
Let $p\ge 2$, $1<\Theta<p+1$, $\theta\in \bR$, and $m\in\{0,1,2,\ldots\}$.  Moreover, let $u\in \mathcal{K}^{1}_{p,\theta,0}(\cD,T)$ be a solution to Equation~\eqref{eq:SHE:Intro} on $\cD$. Then  
\begin{align*}
\E&\int_0^T \sum_{\abs{\alpha}\leq m+1} \int_{\cD} \Abs{\rho^{\abs{\alpha}-1}D^\alpha u}^p \rho_{\circ}^{\theta-2}\ssgrklam{\frac{\rho}{\rho_{\circ}}}^{\Theta-2}\,dx\,dt\\
&\leq
N\,
\E \int_0^T \int_{\cD} \ssgrklam{\Abs{\rho_{\circ}^{-1}
u}^p 
+ 
\sum_{\abs{\alpha}\leq (m-1)\vee 0}\Abs{\rho^{\abs{\alpha}+1}D^{\alpha}f^0}^p+
 \sum_i\sum_{\abs{\alpha}\leq m}\Abs{\rho^{\abs{\alpha}}D^{\alpha}f^i}^p
\\
&\qquad\qquad\qquad\qquad\qquad\qquad\qquad
+
\sum_{\abs{\alpha}\leq m} \Abs{\rho^{\abs{\alpha}}D^\alpha g}_{\ell_2}^p}\rho_{\circ}^{\theta-2}\ssgrklam{\frac{\rho}{\rho_{\circ}}}^{\Theta-2}\,dx\,dt,
\end{align*}
where $N=N(p,\theta,\Theta,\kappa_0,m)$.
In particular, $N$ does not depend on $T$.
\end{thm}

\begin{remark}\label{explanation of key}
As mentioned in~\cite[Remark~2.7]{Kim2004}, the restriction $1<\Theta<p+1$ on the parameter $\Theta$ from Theorem~\ref{lem:estim:2DCone:Theta} is necessary in order to obtain  the corresponding  result for the stochastic heat equation on $\cont^1$ domains with  $\rho_0:=1$. 
Therefore, since a solution to Equation~\eqref{eq:SHE:Intro:a} on $\cD$ that vanishes near the vertex can be considered as a solution to the same equation on a suitable $\cont^1$ domain, the range of $\Theta$ in Theorem~\ref{lem:estim:2DCone:Theta} is sharp.

\end{remark}

In the proof of Lemma~\ref{lem:fi:estim:2DCone} below we first establish existence for equations with nice forcing terms and extend it to the general case with a limit argument.
For this step to succeed, in particular to make sure that we maintain equality `for all $t\leq T$' and therefore the limit also fulfills the equation in the sense of distributions, see Definition~\ref{defn sol}, we need the following lemma.
It also plays a crucial role in the proof of the existence result  for the stochastic heat equation on polygonal domains (Theorem~\ref{thm polygon main}), as it is one of the main ingredients in the proof of Lemma~\ref{lem for gronwall}, which in turn is used for a Gronwall argument to establish existence. 
Its proof is given in detail in Section~\ref{4}.

\begin{lemma}\label{lem 4.5.1}
Let $p\geq 2$ and $\theta\in \bR$. Assume that $u\in \cK^1_{p,\theta,0}(\cD,T)$  satisfies
\begin{equation}\label{disti}
du=(f^0+f^i_{x^i})\,dt+g^k \,dw^k_t, \quad t\in(0,T],
\end{equation}
in the sense of distributions.  
Then $u\in L_p(\Omega; \cont([0,T];L^{[\circ]}_{p,\theta}(\cD)))$ and
\begin{equation}\label{eq:estim:sup:2DCone}
\begin{alignedat}{1}
\E \sgeklam{\sup_{t\leq T} \|u(t,\cdot)\|^p_{L^{[\circ]}_{p,\theta}(\cD)}}
\leq
& N  
\Big( \|u\|^p_{\bwso^{1}_{p,\theta-p}(\cD,T)}+\|f^0\|^p_{\bL^{[\circ]}_{p,\theta+p}(\cD,T)}\\
&\qquad\qquad+
\sum_i\|f^i\|^p_{\bL^{[\circ]}_{p,\theta}(\cD,T)}
+
\|g\|^p_{\bL^{[\circ]}_{p,\theta}(\cD,T;\ell_2)}\Big),
\end{alignedat}
\end{equation}
where $N=N(p,\theta,\kappa_0,T)$.
\end{lemma}

We have now all ingredients needed to state and prove our main existence and uniqueness result for Equation~\eqref{eq:SHE:Intro} on $\cD$.
The representation formula therein uses the Green function for the heat equation on $\cD=\cD_{\kappa_0}$ with zero Dirichlet boundary condition, which we denote by $\Gamma$, see, e.g., \cite[Section~1]{KozRos2012b} for a precise definition.

\begin{thm}[Existence and uniqueness/angular domains]\label{thm:ex:uni:2DCone}
Let $p\geq 2$ and let $\theta\in\bR$ fulfill
\begin{equation}\label{eq:range:vertex}
p \left(1-\frac{\pi}{\kappa_0}\right)<\theta<p\left(1+\frac{\pi}{\kappa_0}\right).
\end{equation}
Assume that $g\in \bL^{[\circ]}_{p,\theta}(\cD,T;\ell_2)$, $f^0\in \bL^{[\circ]}_{p,\theta+p}(\cD,T)$ and $f^i\in \bL^{[\circ]}_{p,\theta}(\cD,T)$, $i=1,2$.
Then
\begin{align*}
u(t,x)
&:=\int_0^t \int_\cD \Gamma(t-s,x,y) f^0(s,y)\,dy\,ds-\sum_{i}\int_0^t \int_\cD \Gamma_{y^i}(t-s,x,y) f^i(s,y)\,dy\,ds
\\
&\qquad+
\sum_{k=1}^\infty
\int_0^t \int_\cD \Gamma(t-s,x,y) g^k(s,y)\,dy\,dw^k_s
\end{align*}
is the unique solution to Equation~\eqref{eq:SHE:Intro} on $\cD$ in the class $\cK^{1}_{p,\theta,0}(\cD,T)$.
Moreover, 
\begin{align}
   \label{main 1}
\|u\|_{\mathbb{K}^1_{p,\theta-p}(\cD,T)}
\leq
N \Big( \|f^0\|_{\bL^{[\circ]}_{p,\theta+p}(\cD,T)}+\sum_i \|f^i\|_{\bL^{[\circ]}_{p,\theta}(\cD,T)}+
\|g\|_{\bL^{[\circ]}_{p,\theta}(\cD,T;\ell_2)}\Big),
\end{align}
where $N=N(p,\theta,\kappa_0)$.
 In particular, $N$ does not depend on $T$.
\end{thm}

If $f^i=0$ for $i=1,2$, then this result has been already proven in \cite[Theorem~3.7]{CioKimLee+2018}.  
The extension presented here is essential to treat Equation~\eqref{eq:SHE:Intro} on polygons in Section~\ref{sec:Polygons} even if the equation on the polygon does not contain terms of this type, see also Remark~\ref{remark 8.24} for details. 

The missing link between \cite[Theorem~3.7]{CioKimLee+2018} and Theorem~\ref{thm:ex:uni:2DCone} is presented in the following lemma.
Its proof relies on Theorem~\ref{lem:estim:2DCone:Theta} and Lemma~\ref{lem 4.5.1}.

\begin{lemma}\label{lem:fi:estim:2DCone}
Let $p\geq 2$ and let $\theta\in \bR$ fulfill~\eqref{eq:range:vertex}.  Assume $f^i\in \bL^{[\circ]}_{p,\theta}(\cD,T)$ for $i=1,2$ and  define
\begin{equation*}
    \label{eqn 4.13.1}
v(t,x):=-\sum_i \int_0^t \int_\cD \Gamma_{y^i}(t-s,x,y) f^i(s,y) \,dy\,ds.
\end{equation*}
Then $v$ is the unique solution in the class $\cK^{1}_{p,\theta,0}(\cD,T)$ to  the equation
\begin{equation}\label{eqn 4.2.1}
dv=(\Delta v + f^i_{x^i})\,dt, \quad t\in (0,T],
\end{equation}
on $\cD$.
Moreover, there exists a constant $N=N(p,\theta,\kappa_0)$ such that 
\begin{equation}
    \label{eqn 4.2.3}
\|v\|_{\mathbb{K}^1_{p,\theta-p}(\cD,T)}
\leq N \sum_i 
\nnrm{f^i}{\bL^{[\circ]}_{p,\theta}(\cD,T)}.
\end{equation}
 In particular, $N$ does not depend on $T$.
\end{lemma}

\begin{proof}
\emph{Step 1.} Let $f^i\in \bL^{[\circ]}_{p,\theta}(\cD,T)$, $i=1,2$. By \cite[Theorem~3.10]{KozNaz2014},  for any  $0<\lambda<\frac{\pi}{\kappa_0}$
($=:\lambda^{\pm}_c$ in \cite{KozNaz2014}),
$$
|\Gamma_y(t-s,x,y)|\leq N \ssgrklam{\frac{\abs{x}}{\abs{x}+\sqrt{t-s}}}^{\lambda} \ssgrklam{\frac{\abs{y}}{\abs{y}+\sqrt{t-s}}}^{\lambda-1} (t-s)^{-\frac{3}{2}} e^{-\frac{\sigma |x-y|^2}{t-s}},
$$
where the constants $\sigma,N>0$ depend only on $\kappa_0$ and $\lambda$.  
Since $\theta$ satisfies~\eqref{eq:range:vertex}, we can take $\lambda$ sufficiently large such that $1-\lambda<\theta/p<1+\lambda$. Then the kernel
$$
\cT_1(t,s,x,y):=\one_{x\in \cD} \one_{y\in \cD} \one_{t>s} |x|^{-1} \frac{|x|^{(\theta-2)/p}}{|y|^{(\theta-2)/p}}\Gamma_y(t-s,x,y)
$$
satisfies the algebraic conditions in \cite[Proposition~A.5]{KozNaz2014} with $\mu=(\theta-2)/p$, $\lambda_1=\lambda_2=\lambda-1$ and $r=1$. Hence by this proposition, 
\[
\|v\|_{\bL^{[\circ]}_{p,\theta-p}(\cD,T)}=\|\rho_{\circ}^{-1}v\|_{\bL^{[\circ]}_{p,\theta}(\cD,T)}\leq N(p,\theta,\kappa_0)\sum_i \|f^i\|_{\bL^{[\circ]}_{p,\theta}(\cD,T)}.
\]

\noindent\emph{Step 2.} Assume the $f^i$s are sufficiently nice, say,  $f^i\in L_p(\Omega_T, \pred;  \cont^2_c(\cD))$. Then by  \cite[Theorem~3.7]{CioKimLee+2018},
$$
v:=\sum_i \int_0^t \int_\cD \Gamma (t-s,x,y) f^i_{x^i}(s,y)\,dy\,ds=-\sum_i \int_0^t \int_\cD \Gamma_{y^i}(t-s,x,y) f^i(s,y)\,dy\,ds
$$
is the unique solution to Equation~\eqref{eqn 4.2.1} in the class $\cK^{1}_{p,\theta,0}(\cD,T)$, see also~\cite{Naz2001, Sol2001}.
This, together with Step~1 and  Theorem~\ref{lem:estim:2DCone:Theta} with $m=0$ and $\Theta=2$ lead to \eqref{eqn 4.2.3} for  $f^i\in L_p(\Omega_T, \pred;  \cont^2_c(\cD))$. 

\noindent\emph{Step 3.} General $f^i\in \bL^{[\circ]}_{p,\theta}(\cD,T)$, $i=1,2$. Uniqueness follows from the case $f^i=0$. Take a sequence $(f^i_n)_{n\in\bN}\subset L_p(\Omega_T, \pred; \cont^2_c(\cD))$ such that
 $f^i_n \to f^i$ in $\bL^{[\circ]}_{p,\theta}(\cD,T)$ for each $i$. Let $v_n \in \cK^{1}_{p,\theta,0}(\cD,T)$ be the solution to Equation~\eqref{eqn 4.2.1}
 with $f^i_n$. Then  by Step~1 and Step~2, $(v_n)$ is a Cauchy sequence in $\mathring{\bwso}^1_{p,\theta-p}(\cD,T)$. 
Let $u:=\lim_{n\to \infty} v_n$ in $\mathring{\bwso}^1_{p,\theta-p}(\cD,T)$. Fix $\varphi \in \cont^{\infty}_c(\cD)$. 
Then taking the limit in
 $$
 (v_n(t,\cdot), \varphi)=-\sum_i \int^t_0((v_n(s,\cdot))_{x^i}+f^i_n(s,\cdot), \varphi_{x^i})ds, \quad \forall\; t\leq T, \quad (\prob\textup{-a.s.})
 $$
and using the continuity of $t\mapsto (u(t),\varphi)$ (due to Estimate~\eqref{eq:estim:sup:2DCone} from Lemma~\ref{lem 4.5.1}), we find that $du=(\Delta u+f^i_{x^i})\, dt$ in the sense of distributions. The integral representation formula for $u$ is due to the fact that by Step~1 we also know that $\lim_{n\to\infty}v_n=v$ in $L_{p,\theta-p}^{[\circ]}(\cD,T)$. Estimate~\eqref{eqn 4.2.3}  follows by taking the limits in the estimates for $v_n$ proven in Step~2.  
\end{proof}

\begin{remark}
Since Lemma~\ref{lem:fi:estim:2DCone} addresses the deterministic heat equation, the restriction $p\geq 2$ is obsolete. The result as well as the proof carry over to the case $p>1$ mutatis mutandis.
\end{remark}

\begin{proof}[Proof of Theorem~\ref{thm:ex:uni:2DCone}]
This is now an immediate consequence of \cite[Theorem~3.7]{CioKimLee+2018} and Lemma~\ref{lem:fi:estim:2DCone} above.
\end{proof}

Theorem \ref{lem:estim:2DCone:Theta} with $\Theta=2$ and Estimate~\eqref{lem:estim:2DCone:Theta}  now lead to the following higher order regularity result of the solution depending on the regularity of the forcing terms $f^{0}$, $f^i$, and $g^k$. 
 Recall that in this section $\rho$ denotes the distance to the boundary of $\cD$.

\begin{corollary}[higher order regularity/angular domains]
\label{high}
 Given the setting of Theorem~\ref{thm:ex:uni:2DCone}, let $u$ be the unique solution in the class $\cK^{1}_{p,\theta,0}(\cD,T)$ to Equation~\eqref{eq:SHE:Intro} on $\cD$. Assume that  
\begin{align*}
C(m,\theta, f^i,f^0,g)
&:=
\E \int^T_0 \int_{\cD}\ssgrklam{\sum_{|\alpha|\leq  (m-1)\vee 0} |\rho^{\abs{\alpha}+1}D^{\alpha}f^0|^p+
 \sum_i\sum_{|\alpha|\leq m} |\rho^{|\alpha|}D^{\alpha}f^i|^p\\
&\quad\quad\qquad\qquad\qquad+|\,\rho_{\circ} f^0|^p+
\sum_{\abs{\alpha}\leq m} |\rho^{\abs{\alpha}}D^\alpha g|_{\ell_2}^p }\rho_{\circ}^{\theta-2}\, dx\,dt <\infty
\end{align*}
for some $m\in\{0,1,2,\ldots\}$. Then
$$
\E\int_0^T \sum_{\abs{\alpha}\leq m+1} \int_{\cD}   \Abs{\rho^{\abs{\alpha}-1}D^\alpha u}^p \abs{x}^{\theta-2}\,dx\,dt\leq N\,C(m,\theta, f^i,f^0,g)<\infty,$$
where $N=N(p,\theta,\kappa_0,m)$.  In particular, $N$ does not depend on $T$.

\end{corollary}

We will need the following `general uniqueness' lemma  to handle the stochastic heat equation on polygons in Section~\ref{sec:Polygons}.

\begin{lemma}
   \label{lem for uniqueness}  
Let $2\leq p_1\leq p_2$ and let $\theta_1, \theta_2\in\bR$ satisfy \eqref{eq:range:vertex} for  $p=p_1$ and $p=p_2$, respectively. Assume for both $j=1$ and $j=2$,
\[
f^0 \in \bL^{[\circ]}_{p_j,\theta_j+p_j}(\cD,T), \quad f^i \in \bL^{[\circ]}_{p_j,\theta_j}(\cD,T),\,\,i=1,2, \quad g\in \bL^{[\circ]}_{p_j,\theta_j}(\cD,T;\ell_2)
\]
 and let $u\in \cK^{1}_{p_1,\theta_1,0}(\cD,T)$ be the solution to  Equation~\eqref{eq:SHE:Intro}.   Then  $u\in \cK^{1}_{p_2,\theta_2,0}(\cD,T)$.     \end{lemma}
     
\begin{proof}
This follows from the integral representation formula of the solution in Theorem \ref{thm:ex:uni:2DCone}, that is, the unique solutions in  
$\cK^{1}_{p_1,\theta_1,0}(\cD,T)$ and $\cK^{1}_{p_2,\theta_2,0}(\cD,T)$ have the same representation formula.
\end{proof}

\begin{remark}
\label{remark cones}
 To keep the presentation short, the results in this section are formulated only for angular domains $\cD\subseteq\bR^2$ with vertex at the origin and with one of the edges being the positive $x^1$-axis.
However, since every angular domain in $\bR^2$ can be seen as a translation of a rotation of such a domain, all results can be extended accordingly, as the Laplace operator is invariant under translations and rotations. 
More precisely,  fix $a\in (-\pi, \pi)$ and $x_0\in \bR^2$.  Let
   \begin{equation*}
\tilde{\cD}:=\tilde{\cD}_{\kappa_0}(x_0,a):=\big\{x\in \bR^2: x=x_0+(r\cos\vartheta,r\sin\vartheta),\; r>0,\;\vartheta\in (a,a+\kappa_0)\big\}.
\end{equation*}
Replacing $\cD$ and $\rho_\circ$ by $\tilde{\cD}_{\kappa_0}(x_0,a)$ and $\tilde{\rho}_\circ(x):=|x-x_0|$, respectively, in the definitions of the weighted Sobolev spaces from above, we can define analogous spaces, such as $K^n_{p,\theta}(\tilde{\cD})$, $\mathbb K^n_{p,\theta}(\tilde{\cD},T)$ and  
$\cK^1_{p,\theta,0}(\tilde{\cD},T)$, on $\tilde{\cD}$.  
Then,  the results in this section  hold with $\tilde{\cD}$ in place of $\cD$.  Indeed,
let $Q=(q_{ij})_{1\leq i,j\leq 2}$ be the orthogonal matrix such that $\tilde{\cD}_{\kappa_0}(x_0,a)=x_0+Q \cD_{\kappa_0}$. Then, since the Laplacian is invariant under the rotations and translations,  the statement that $u\in \cK^1_{p,\theta,0}(\tilde{\cD},T)$  satisfies
\begin{equation}
\label{eqn 4.10.7}
du=(\Delta u +f^0+f^i_{x^i})\,dt+g^k dw^k_t,
\end{equation}
 in the sense of distribution (analogous meaning to Definition~\ref{defn sol})  is the same as the statement that $v(t,x):=u(t,x_0+Qx)\in   \cK^1_{p,\theta,0}(\cD,T)$  satisfies
\[
dv=(\Delta v +\tilde{f}^0+\tilde{f}^i_{x^i})\, dt+\tilde{g}^k \, dw^k_t,
\]
where
$
\tilde{f}^0(t,x)=f^0(t,x_0+Qx)$, $\tilde{f}^i(t,x)=q_{1i}f^1(t,x_0+Qx)+q_{2i}f^2(t,x_0+Qx)$, $i=1,2$,
and $\tilde{g}(t,x)=g(t,x_0+Qx)$. 
Hence, all existence and uniqueness results as well as all estimates  can be extended to general angular domains, since, obviously,
$$
\|h(x)\|_{K^n_{p,\theta}(\tilde{\cD})} \sim \|h(x_0+Qx)\|_{K^n_{p,\theta}(\cD)}
$$
for any $h\in K^n_{p,\theta}(\tilde{\cD})$.  
 To extend Lemma~\ref{lem 4.5.1}, formally set $\Delta u=0$ in~\eqref{eqn 4.10.7}.
\end{remark}

\mysection{Proof of Theorem  \ref{lem:estim:2DCone:Theta}}\label{sec:proof:lift}

In this section we give a detailed proof of the key estimate from Theorem~\ref{lem:estim:2DCone:Theta}.
Our proof is based on a suitable a-priori estimate for the stochastic heat equation on $\cont^1$ domains, as presented in Lemma~\ref{lem 10} below.
We use this result to establish an estimate for the solution on a subdomain of $\cD$ which is bounded away from the vertex and from infinity (see Lemma~\ref{lem:estim:2DCone:U1} below).
Then we can prove Theorem~\ref{lem:estim:2DCone:Theta} by using a dilation argument, as $\cD$ is invariant under positive dilation.
For this strategy to succeed, it is crucial that the constant in Lemma~\ref{lem 10} does not depend on the time horizon $T$.

We start with the definition of the weighted Sobolev spaces $H^n_{p,\Theta}(G)$ on $\cont^1$  domains $G\subseteq\bR^d$ ($d\geq 1$), which we need for the statement of Lemma~\ref{lem 10}. First we recall the definition of a $\cont^1$ domain. 
%Let  $\pi_0$ be a nonnegative function on $[0,\infty)$ such that $\pi_0(t)\to 0$ as $t \downarrow 0$.

\begin{defn}\label{definition domain}
  Let $G$ be a  domain in $\bR^d$, $d\geq 1$.
We write  $\partial G\in \cont^1_u$  and say that $G$ is a  $\cont^1$ domain  if there  exist constants $r_0, K_0\in(0,\infty)$ such that 
for any  $x_0 \in \partial G$ there exists
 a one-to-one continuously differentiable mapping $\Psi$ of
 $B_{r_0}(x_0)$ onto a domain $J\subset\bR^d$ such that
\begin{enumerate}[align=right,label=\textup{(\roman*)}] 
\item $J_+:=\Psi(B_{r_0}(x_0) \cap G) \subset \bR^d_+$ and
$\Psi(x_0)=0$;

\item $\Psi(B_{r_0}(x_0) \cap \partial  G)= J \cap \{y\in
\bR^d:y^1=0 \}$;

\item $\|\Psi\|_{\cont^{1}(B_{r_0}(x_0))}  \leq K_0 $ and
$|\Psi^{-1}(y_1)-\Psi^{-1}(y_2)| \leq K_0 |y_1 -y_2|$ for any $y_i
\in J$;

\item  $\Psi_x$ is uniformly continuous in $B_{r_0}(x_0)$.

% C^1_u case
%\item  $|\Psi_{x}(x_1)-\Psi_x(x_2)|\leq \pi_0(|x_1-x_2|)$ for any $x_i\in B_{r_{0}}(x_{0})$.
 \end{enumerate}
\end{defn}
Throughout this article, we assume that $G$ is  either $\bR^d_+:=\{x\in \bR^d\colon x^1>0\}$ or a bounded $\cont^1$ domain in $\bR^d$ ($d\geq 1$). Note that in both cases, $G$ is of class $\cont^1_u$ in the sense of \cite[Assumption~2.1]{Kim2004}.
Recall that $\rho(x)=\rho_G(x)=\mathrm{dist}(x,\partial G)$ for $x\in G$;
$\rho(x)=x^1$ if $G=\bR^2_+$. 
For $p>1$ and $\Theta\in\bR$, we write
\[
L_{p,\Theta}(G):=L_{p}(G,\rho^{\Theta-d} dx;\R) \quad 
\text{and}
\quad
L_{p,\Theta}(G;\ell_2):=L_p(G,\rho^{\Theta-d} dx;\ell_2)
\]
for the weighted $L_p$-spaces of real-valued/$\ell_2$-valued functions with weight $\rho^{\Theta-d}$.
For $n\in  \{0,1,2,\ldots\}$, by $H^n_{p,\theta}(G)$ we denote the space of all 
$f\in L_{p,\Theta}(G)$  such that
\begin{equation}
     \label{eqn 4.9.5}
\|f\|^p_{H^n_{p,\Theta}(G)}:=\sum_{|\alpha|\leq n}  \|\rho^{\abs{\alpha}} D^\alpha f\|^p_{L_{p,\Theta}(G)}<\infty.
\end{equation}
Moreover, we define the dual spaces
\[
H^{-n}_{p,\Theta}(G):=\grklam{H^n_{p',\Theta'}(G)}^*,\qquad\frac{1}{p}+\frac{1}{p'}=1,\quad \frac{\Theta}{p}+\frac{\Theta'}{p'}=d.
\]
The space $H^n_{p,\Theta}(G;\ell_2)$ is defined analogously for $n\in \bZ$.

To state the main properties of these spaces, we introduce some additional notation. 
For $k\in\{0,1,2,\ldots\}$, let
$$
|f|^{(0)}_{k}:=|f|^{(0)}_{k,G} :=\sup_{\substack{x\in G\\
|\beta| \leq k}}\rho^{|\beta|}(x)|D^{\beta}f(x)|.
$$
If $G$ is bounded, let $\psi$ be a bounded $\cont^\infty$ function defined in $G$ with $|\psi|^{(0)}_k+|\psi_x|^{(0)}_k<\infty$ for any $k$, which is comparable to $\rho$, i.e., $N^{-1}\rho(x)\leq \psi(x)\leq N \rho(x)$ for some constant $N>0$; see, e.g., \cite[ Section~2]{KimKry2004}.
It is known that, if $G$ is bounded, then  the map $\Psi$ in Definition~\ref{definition domain} can be chosen in such a way that $\Psi$ is infinitely differentiable in $B_{r_0}(x_0)\cap G$ and for any multi-index $\alpha$
\begin{equation}
                \label{eqn 12.4.9}
\sup_i \sup_{B_{r_0}(x_0) \cap G} \rho^{|\alpha|}|D^{\alpha}\Psi_{x^i}|\leq N(\alpha)<\infty;
 \end{equation}
see, e.g., \cite{KimKry2004} or the proof of \cite[Lemma~4.9]{KimLee2011}.
Actually, after appropriate rotation and translation, one can take $\Psi(x^1,x')=(\psi(x),x')$.   By \cite[Theorem~3.2]{Lot2000} and \eqref{eqn 12.4.9} above,   if  $\text{supp}\, u\subset B_{r}(x_0)\cap \overline{G}$ and $r<r_0/{K_0}$,  then for any $\nu, \Theta \in \bR$ and $n\in \bZ$, we have
 \begin{equation}\label{eqn 12.7.1}
 \|\psi^{\nu}u\|_{H^{n}_{p,\Theta}(G)} \sim \|(x^1)^{\nu}(u\circ\Psi^{-1})\|_{H^{n}_{p,\Theta}(\bR^d_+)}.
 \end{equation}

Here are some  other properties of the spaces $H^{n}_{p,\Theta}(G)$ taken from~\cite{Lot2000} (see also \cite{KimKry2004,Kry1999c}). If $G=\bR^d_+$, let $\psi(x):=\rho(x)=x^1$.

\begin{lemma}\label{collection}
\begin{enumerate}[leftmargin=*,label=\textup{(\roman*)}, wide] 
\item  $\cont^\infty_c(G)$ is dense in $H^n_{p,\Theta}(G)$.

\item\label{col:psiD} For any $n\in \bZ$ the operators $\psi D, D \psi: H^{n}_{p,\Theta}(G)\to
H^{n-1}_{p,\Theta}(G)$ are bounded linear operators. In fact, 
for any $u\in H^{n}_{p,\Theta}(G)$,
$$
\|u\|_{H^{n}_{p,\Theta}(G)} \leq N\|\psi
u_{x^i}\|_{H^{n-1}_{p,\Theta}(G)}+N\|u\|_{H^{n-1}_{p,\Theta}(G)}
\leq N \|u\|_{H^{n}_{p,\Theta}(G)},
$$
$$
\|u\|_{H^{n}_{p,\Theta}(G)} \leq N\|(\psi
u)_{x^i}\|_{H^{n-1}_{p,\Theta}(G)}+N\|u\|_{H^{n-1}_{p,\Theta}(G)}
\leq N \|u\|_{H^{n}_{p,\Theta}(G)}
$$
hold, where $N$ is independent of $u$ and  $i\in\{1,\ldots,d\}$.

\item\label{col:psi:theta} For any $\nu\in \bR$, $n\in \bZ$,
$\psi^{\nu}H^{n}_{p,\Theta}(G)=H^{n}_{p,\Theta-p\nu}(G)$
and
\begin{equation}
                            \label{open}
\|u\|_{H^{n}_{p,\Theta-p\nu}(G)} \leq N
\|\psi^{-\nu}u\|_{H^{n}_{p,\Theta}(G)}\leq
N\|u\|_{H^{n}_{p,\Theta-p\nu}(G)}.
\end{equation}

\item\label{col:multiplier}  For $\Theta\in \bR$ and $n\in \bZ$,
$$
\|a u\|_{H^{n}_{p,\Theta}(G)}\leq
N(d,n)|a|^{(0)}_{|n|}\|u\|_{H^{n}_{p,\Theta}(G)}.
$$

\item\label{col:bounded:embedding} If $G$ is bounded and $\Theta_1<\Theta_2$, then $H^n_{p,\Theta_1}(G)\subset H^n_{p,\Theta_2}(G)$ and
$$
\|u\|_{H^n_{p,\Theta_2}(G)}\leq N(n,d,\Theta_1,\Theta_2) \|u\|_{H^n_{p,\Theta_1}(G)}.
$$

\item\label{col:cptsupport} Let $n\in \bZ$  and $u\in H^{n}_{p,\Theta}(G)$ with  $K:=\mathrm{supp}\, u  \subset G$.  Then for some $N=N(d,p,\Theta,K)>0$,
$$
N^{-1}\|u\|_{H^n_p(\bR^d)}\leq \|u\|_{H^n_{p,\Theta}(G)}\leq N\|u\|_{H^n_p(\bR^d)},
$$
where $H^n_p(\bR^d):=\{u: D^{\alpha}u\in L_p(\bR^d), \, \forall\, |\alpha|\leq n\}$ if $n\geq 0$, and otherwise  it is the dual space of $H^{-n}_q(\bR^d)$, where $\frac1p+\frac1q=1$.
\end{enumerate}

\end{lemma}

Note that, by Lemma~\ref{collection}\ref{col:multiplier} and the properties of $\psi$, $\psi$ is a point-wise multiplier in $H^n_{p,\Theta}(G)$   if $G$ is bounded.

 For the corresponding spaces of predictable $H^n_{p,\Theta}(G)$/$H^n_{p,\Theta}(G;\ell_2)$-valued stochastic processes we use the abbreviations
$$
\bH^{n}_{p,\Theta}(G,T)
:=
L_p(\Omega_T, \pred;H^{n}_{p,\Theta}(G))
\quad 
\text{and}\quad\bL_{p,\Theta}(G,T):=\bH^0_{p,\Theta}(G,T),
$$
as well as
\[
\bH^n_{p,\Theta}(G,T;\ell_2)
:=
L_p(\Omega_T, \pred;H^{n}_{p,\Theta}(G;\ell_2))\quad 
\text{and}\quad\bL_{p,\Theta}(G,T;\ell_2):=\bH^0_{p,\Theta}(G,T;\ell_2).
\]
 The following classes of stochastic processes are tailor-made for the analysis of Equation~\eqref{eq:SHE:Intro} on $G$.
\begin{defn}
For $p\geq 2$ and $\Theta\in\bR$ we write $u\in\frH^n_{p,\Theta,0}(G,T)$  if
$u\in\bH^n_{p,\Theta-p}(G,T)$
 and
there exist 
$f\in \bH^{n-2}_{p,\Theta+p}(G,T)$ and
 $g\in \bH^{n-1}_{p,\Theta}(G,T;\ell_2)$
such that 
\begin{equation*}\label{eqn 28_1}
 du=f\, dt +g^k \, dw^k_t,\quad t\in (0,T],
\end{equation*}
on $G$ in the sense of distributions with $u(0,\cdot)=0$; see Definition~\ref{defn sol} accordingly.  We denote 
$$
\bD u:=f\qquad\text{and}\qquad \bS u :=g.
$$

\end{defn}

 In this article, Equation~\eqref{eq:SHE:Intro} has the following meaning on $G$.
\begin{defn}\label{defn sol:C1}
We say that  $u\in\bH^{n}_{p,\Theta-p}(G,T)$ is
a solution to Equation~\eqref{eq:SHE:Intro} on $G$ 
in the class $\frH^{n}_{p,\Theta,0}(G,T)$
if
$u\in \frH^{n}_{p,\Theta,0}(G,T)$ with
\[
\bD u = \Delta u + f^0+f^i_{x^i}
\qquad
\text{and}
\qquad
\bS u = g.
\]
\end{defn}

\begin{remark}
All definitions above are given only for $\cont^1$ domains, as we say from the beginning that in this article $G$ is either a bounded $\cont^1$ domain or the half plane. However,  all the spaces defined above as well as the solution concept make sense on any domain $\domain\subset\bR^d$ with non-empty boundary.
\end{remark}

 Now we have all notions we need in order to state and prove the a-priori estimate for Equation~\eqref{eq:SHE:Intro} on bounded $\cont^1$ domains that we use to prove Lemma~\ref{lem:estim:2DCone:U1} and therefore Theorem~\ref{lem:estim:2DCone:Theta}.

\begin{lemma}\label{lem 10}
Let $G\subset\bR^d$ be a bounded $\cont^1$ domain, $p\geq2$, $n\in\{-1,0,1,\ldots\}$, and $d-1< \Theta<d-1+p$. Moreover, let $f^0\in \bH^{n}_{p,\Theta+p}(G,T)$, $f^i\in \bH^{n+1}_{p,\Theta}(G,T)$, $i=1,2$, and $g\in
\bH^{n+1}_{p,\Theta}(G,T;\ell_2)$. 
Assume  $u$ is a solution to Equation~\eqref{eq:SHE:Intro} on $G$
in the class $\frH^{1}_{p,\Theta_1,0}(G,T)$ for some $\Theta_1\in [\Theta, d+p-1)$. Then
$u\in \frH^{n+2}_{p,\Theta,0}(G,T)$ and
\begin{equation}\label{lem 10:estimate}
\begin{alignedat}{1}
\|u\|_{\bH^{n+2}_{p,\Theta-p}(G,T)} 
&\leq N 
\Big(
\|u\|_{\bL_{p,\Theta}(G,T)}
+
\|f^0\|_{\bH^{n}_{p,\Theta+p}(G,T)}\\
&\quad\qquad\qquad \quad + \sum_i\|f^i\|_{\bH^{n+1}_{p,\Theta}(G,T)}+
\|g\|_{\bH^{n+1}_{p,\Theta}(G,T;\ell_2)}\Big), 
\end{alignedat}
\end{equation}
where $N=N(p,d, \Theta, n,G)$.  In particular, $N$ does not depend on $T$.
\end{lemma}

\begin{proof}
\emph{Step 1.} First we prove that $u\in \frH^{n+2}_{p,\Theta}(G,T)$. 
By~
\cite[Theorem~2.9]{Kim2004}, under the given assumptions, there exists a solution $v\in \frH^{n+2}_{p,\Theta}(G,T)$.  Since $G$ is bounded, by Lemma~\ref{collection}\ref{col:bounded:embedding}, 
$$
\bH^{n+2}_{p,\Theta-p}(G,T)\subset \bH^{1}_{p,\Theta-p}(G,T)\subset\bH^1_{p,\Theta_1-p}(G,T),
$$
and therefore $v\in \frH^1_{p,\Theta_1}(G,T)$.  
By the uniqueness part of \cite[Theorem~2.9]{Kim2004} we get $u=v$ (in 
$\frH^1_{p,\Theta_1}(G,T)$).

\noindent\emph{Step 2.}  We prove Estimate~\eqref{lem 10:estimate}.  
In fact, by \cite[Theorem~2.9]{Kim2004}, this estimate holds even without the term $\|u\|_{\bL_{p,\Theta}(G,T)}$ on the right hand side if we allow  the constant $N$ to depend on $T$. 
However, a close look at the proof of \cite[Theorem~2.9]{Kim2004} reveals that, indeed, if we leave this term on the right hand side, the constant can be kept independent of $T$, since the dependence on $T$ comes in only in the very last step of the relevant part of the proof of \cite[Theorem~2.9]{Kim2004}, when a Gronwall argument is used in order to get rid of the terms that depend on $u$ on the right hand side.

Instead of reproving~\cite[Theorem~2.9]{Kim2004}, we illustrate the relevant steps in the proof and the changes required to obtain independence of $T$.
The key estimate is (5.6) of ~\cite{Kim2004}, which says that 
\begin{equation}\label{eqn 4.7.1}
\begin{alignedat}{1}
  \|u\|_{\bH^{n+2}_{p,\Theta-p}(G,T)}&\leq N\sgrklam{\|\psi u_{x}\|_{\bH^n_{p,\Theta}(G,T)}+\|u\|_{\bH^n_{p,\Theta}(G,T)}\\
  &\quad \qquad\qquad\quad +\|\psi(f^0+f^i_{x^i})\|_{\bH^n_{p,\Theta}(G,T)}+\|g\|_{\bH^{n+1}_{p,\Theta}(G,T;\ell_2)} }. 
\end{alignedat}
\end{equation}
In our setting, i.e., for the stochastic heat equation, the constant $N$ in this estimate does not depend on $T$.
Indeed, as  explained in detail in the proof of \cite[Theorem~2.9]{Kim2004}, by using a suitable partition of unity, Estimate~\eqref{eqn 4.7.1} is obtained through a combination of an a-priori estimate on the half space (\cite[Theorem~2.10]{Kim2004}) and its analogue on the entire space (\cite[Theorem~5.1]{Kry1999}). 
The former theorem does not add any dependence on $T$ as the constant therein is explicitly proven to be independent of $T$.
This is different for \cite[Theorem~5.1]{Kry1999}: The constant therein may indeed depend on $T$. 
However, this dependence only occurs if we consider equations with variable coefficients. 
For the stochastic heat equation we may use the a-priori estimate from~\cite[Theorem~4.2]{Kry1999} instead, which holds with a constant that does not depend on $T$.
Note that on the left hand side of the estimate in~\cite[Theorem~4.2]{Kry1999}, we have the $L_p$-norm of the second order derivatives of the solution.
However, this is not a problem since in the proof of~\eqref{eqn 4.7.1} we only use this estimate for the solution of a stochastic heat equation with compact support in $\bR^d$ and, due to Poincar\'e's inequality, for a function $v$ with compact support in $\bR^d$, the norms $\|v\|_{H^{n+2}_p}$, $\sum_i\|v_{x^i}\|_{H^{n+1}_p}$ and $\sum_{i,j}\|v_{x^ix^j}\|_{H^n_p}$ are all equivalent.

To derive~\eqref{lem 10:estimate} from \eqref{eqn 4.7.1} we argue as follows: Using  \eqref{eqn 4.7.1} and the basic properties of the weighted Sobolev spaces $H^n_{p,\Theta}(G)$ from Lemma~\ref{collection}, we easily obtain
\begin{equation}\label{eqn 4.7.2}
\begin{alignedat}{1}
  \|u\|_{\bH^{n+2}_{p,\Theta-p}(G,T)}
  &\leq N \sgrklam{\|u\|_{\bH^{n+1}_{p,\Theta}(G,T)}+\|f^0\|_{\bH^n_{p,\Theta+p}(G,T)}\\
  &\quad \qquad\qquad\qquad\quad +\sum_i\|f^i\|_{\bH^{n+1}_{p,\Theta}(G,T)}+\|g\|_{\bH^{n+1}_{p,\Theta}(G,T;\ell_2)} }. 
\end{alignedat}
\end{equation}
Thus if  $n=-1$, then \eqref{lem 10:estimate} is proved. If $n\geq 0$, then another application of Lemma~\ref{collection}\ref{col:bounded:embedding} shows that~\eqref{eqn 4.7.2} implies
\begin{equation*}
\begin{alignedat}{1}
  \|u\|_{\bH^{n+2}_{p,\Theta-p}(G,T)}
  &\leq N \sgrklam{\|u\|_{\bH^{n+1}_{p,\Theta-p}(G,T)}+\|f^0\|_{\bH^n_{p,\Theta+p}(G,T)}\\
  &\quad \qquad\qquad\qquad\quad +\sum_i\|f^i\|_{\bH^{n+1}_{p,\Theta}(G,T)}+\|g\|_{\bH^{n+1}_{p,\Theta}(G,T;\ell_2)} },
\end{alignedat}
\end{equation*}
which means that we can control $\|u\|_{\bH^{n+2}_{p,\Theta-p}(G,T)}$ by $\|u\|_{\bH^{n+1}_{p,\Theta-p}(G,T)}$ and suitable norms of the free terms.
After repeating this step for $n$ more times, we arrive at 
\begin{equation*}
\begin{alignedat}{1}
  \|u\|_{\bH^{n+2}_{p,\Theta-p}(G,T)}
  &\leq N \sgrklam{\|u\|_{\bH^{1}_{p,\Theta-p}(G,T)}+\|f^0\|_{\bH^n_{p,\Theta+p}(G,T)}\\
  &\quad \qquad\qquad\qquad\quad +\sum_i\|f^i\|_{\bH^{n+1}_{p,\Theta}(G,T)}+\|g\|_{\bH^{n+1}_{p,\Theta}(G,T;\ell_2)} }.
\end{alignedat}
\end{equation*}
Estimate~\eqref{lem 10:estimate} follows by applying~\eqref{eqn 4.7.2} with $n=-1$. Note that all constants in the estimates above are independent of $T$.
  \end{proof}

Now we go back to $\cD\subset\bR^2$ and  proceed to prove Theorem~\ref{lem:estim:2DCone:Theta}. The key step is presented in Lemma~\ref{lem:estim:2DCone:U1} below. It provides an estimate of suitable weighted $L_p$-norms of the derivatives of the solution in 
\[
U_1:=\{x\in\cD\,\colon\,1<\abs{x}<4\}
\]
by appropriate weighted $L_p$-norms of $u$ and of the derivatives of the free terms on the slightly bigger domain
\[
V_1:=\{x\in\cD\,\colon\, 1/2<\abs{x}<8\}.
\]
As these domains are bounded away from the vertex, the estimate involves only the distance $\rho$ to the boundary.
It is crucial that the constant in Lemma~\ref{lem:estim:2DCone:U1} below does not depend on $T$. In the proof we use Lemma \ref{lem 10} with $d=2$.

\begin{lemma}\label{lem:estim:2DCone:U1}
 Let $p\ge 2$, $1<\Theta<p+1$, and $m\in\{0,1,2,\ldots\}$.  Moreover, let $u\in \mathcal{K}^{1}_{p,\theta,0}(\cD,T)$ be a solution to Equation~\eqref{eq:SHE:Intro} on $\cD$ for some $\theta\in\bR$. Then  
\begin{equation}\label{eqn 4.9.3}
\begin{alignedat}{1}
\E&\int_0^T \sum_{\abs{\alpha}\leq m+1} \int_{U_1} \Abs{\rho^{\abs{\alpha}-1}D^\alpha u}^p\rho^{\Theta-2}\,dx\,dt\\
&\leq
N\,
\E \int_0^T \int_{V_1} \ssgrklam{\abs{
u}^p 
+ 
\sum_{\abs{\alpha}\leq (m-1)\vee 0}\Abs{\rho^{\abs{\alpha}+1}D^{\alpha}f^0}^p+
 \sum_{\abs{\alpha}\leq m}\sum_i\Abs{\rho^{\abs{\alpha}}D^{\alpha}f^i}^p  
\\
&\qquad\qquad\qquad\qquad\qquad\qquad\qquad\qquad
+
\sum_{\abs{\alpha}\leq m} \Abs{\rho^{\abs{\alpha}}D^\alpha g}_{\ell_2}^p}\rho^{\Theta-2}\,dx\,dt,  
\end{alignedat}
\end{equation}
 where $N=N(p,\Theta,\kappa_0,m)$. In particular, $N$  does not depend on $T$.
\end{lemma}

\begin{proof}
Assume that the integrals on the right hand side of~\eqref{eqn 4.9.3} are finite (if not, the statement is trivial). 
Fix a constant $\varepsilon \in (0,1/4)$, and  for $k=1,2,3$, let
$$U^k_1:=\{x\in \cD: 2^{-k\varepsilon}<|x|<2^{2+k\varepsilon}\}.
$$
Choose a $\cont^{\infty}$ radial non-negative  function $\eta=\eta(|x|)$ such that $\eta(t)=1$ for $t\in [1,4]$ and $\eta(t)=0$ if $t\not\in [2^{-\varepsilon}, 2^{2+\varepsilon}]$. Also choose a $\cont^{1}$ domain $G\subset\cD$ such that
$$
U^2_1 \subset G \subset U^3_1\subset V_1.
$$
By the choice of $\eta$ and $G$, $u\eta$ vanishes on the boundary of $G$ and there exists $N=N(\varepsilon)$ such that for all $x\in \cD \cap  \text{supp}\,\eta$,
\begin{equation}
 \label{rho}
N^{-1} \rho_{\cD}(x)\leq \rho_G (x) \leq N \rho_{\cD}(x).
\end{equation}
 Let $l:=2m+3$ and set $\zeta:=\eta^l$.  
Then
$$
d(\zeta u)=\grklam{\Delta (\zeta u)+u\Delta \zeta-f^i\zeta_{x^i}+f^0\zeta+(-2u\zeta_{x^i}+f^i\zeta)_{x^i}}dt+\zeta g^k dw^k_t, \quad t\in(0,T], 
$$
on $G$ in the sense of distributions.
 Moreover, $\zeta u\in\frH^1_{p,2,0}(G,T)$ since $u\in\cK^1_{p,\theta,0}(\cD,T)$ solves Equation~\eqref{eq:SHE:Intro} on $\cD$ and since, by Hardy's inequality, 
\begin{align*}
\nnrm{\zeta u}{L_{p,2-p}(G)}
\leq
N\,\nnrm{(\zeta u)_x}{L_{p,2}(G)}
\leq
N\,\nnrm{\zeta u}{K^1_{p,\theta-p}(\cD)}
\leq
N(p,\theta,G,\zeta)\,\nnrm{ u}{K^1_{p,\theta-p}(\cD)}.
\end{align*}
The second inequality above is due to $\rho_0 \sim 1$ on the support of $\zeta$.
Thus, by an application of Lemma~\ref{lem 10} with $n=m-1$, we obtain 
\begin{equation*}
\begin{alignedat}{1}
\|\zeta u\|_{\bH^{m+1}_{p,\Theta-p}(G,T)}
&\leq
 N \sgrklam{\sum_i\|\zeta f^i\|_{\bH^{m}_{p,\Theta}(G,T)}+\sum_i\|u\zeta_{x^i}\|_{\bH^m_{p,\Theta}(G,T)}\\
&\qquad\qquad\qquad + \|\zeta f^0\|_{\bH^{m-1}_{p,\Theta+p}(G,T)} 
 +\|u\Delta \zeta\|_{\bH^{m-1}_{p,\Theta+p}(G,T)}\\ 
&\qquad\qquad\qquad\qquad\quad + \|\zeta g\|_{\bH^m_{p,\Theta}(G,T;\ell_2)}
+ \|\zeta u\|_{\bL_{p,\Theta}(G,T)} },
\end{alignedat}
\end{equation*}
 once we can prove that all norms on the right hand side are finite. 
The norms that do not involve $u$ together with $\|\zeta u\|_{\bL_{p,\Theta}(G,T)}$ can be estimated by
\begin{align*}
\|\eta u\|_{\bL_{p,\Theta}(G,T)}
+
\|\eta f^0\|_{\bH^{(m-1)\vee 0}_{p,\Theta+p}(G,T)} 
+
\|\eta g\|_{\bH^m_{p,\Theta}(G,T;\ell_2)}
+
\sum_i\|\eta f^i\|_{\bH^{m}_{p,\Theta}(G,T)},
\end{align*}
which is finite since the right hand side of~\eqref{eqn 4.9.3} is finite (use~\eqref{rho}, Lemma~\ref{collection} (in particular, part~\ref{col:multiplier}) and the properties of $\eta$ to estimate the norms above by the right hand side of~\eqref{eqn 4.9.3}).
Moreover,  since  
$$
\|u\zeta_{x^i}\|_{H^m_{p,\Theta}(G)}
=
l \,\|u \eta^{l-1}\eta_{x^i}\|_{H^{m}_{p,\Theta}(G)}\leq  N\|u\eta^{l-2}\|_{H^m_{p,\Theta}(G)},
$$
and 
$$
\|u \Delta \zeta\|_{H^{m-1}_{p,\Theta+p}(G)}\leq N\|u\eta^{l-2}\|_{H^m_{p,\Theta}(G)},
$$
the condition
\begin{equation}
    \label{eqn 4.15.1}
\|\eta^{l-2} u\|_{\bH^m_{p,\Theta}(G,T)}<\infty
\end{equation}
is sufficient in order be able to apply Lemma~\ref{lem 10} and obtain
\begin{align*}
\|\eta^l u\|_{\bH^{m+1}_{p,\Theta-p}(G,T)}
&\leq
N \sgrklam{
\|\eta^{l-2} u\|_{ \bH^m_{p,\Theta}(G,T)}
+ \|\eta u\|_{\bL_{p,\Theta}(G,T)}  \\
&\quad\,\,+ \sum_i\|\eta f^i\|_{\bH^{m}_{p,\Theta}(G,T)}
+ \|\eta f^0\|_{\bH^{m-1}_{p,\Theta+p}(G,T)} 
+ \|\eta g\|_{\bH^m_{p,\Theta}(G,T;\ell_2)}}.\nonumber
\end{align*}
In particular, 
this shows that 
\[
\nnrm{\eta^{l-2} u}{\bH^{m}_{p,\Theta}(G,T)}<\infty
\quad\Rightarrow\quad
\nnrm{\eta^l u}{\bH^{m+1}_{p,\Theta-p}(G,T)}<\infty.
\]
In order to prove that~\eqref{eqn 4.15.1} holds, we argue as follows: 
Since $\|\eta^{l-2} u\|_{H^m_{p,\Theta}(G)}\leq N \|\eta^{l-2} u\|_{H^m_{p,\Theta-p}(G)}$,
we can iterate the arguments above with $m$ replaced by $m-j$ and $l$ replaced by $l-2j$ successively for $j=1,\ldots,m$. After finitely many steps we arrive at the statement that if 
$
\|\eta u\|_{\bL_{p,\Theta}(G,T)}<\infty
$, then
\begin{equation}\label{eqn 4.9.3.2}
\begin{alignedat}{1}
\|\eta^{2m+3} u\|_{\bH^{m+1}_{p,\Theta-p}(G,T)}
&\leq
N \sgrklam{
\|\eta u\|_{\bL_{p,\Theta}(G,T)} + \sum_i\|\eta f^i\|_{\bH^{m}_{p,\Theta}(G,T)}  \\
&\quad\quad\qquad
+ \|\eta f^0\|_{\bH^{(m-1)\vee 0}_{p,\Theta+p}(G,T)} 
+ \|\eta g\|_{\bH^m_{p,\Theta}(G,T;\ell_2)}}.
\end{alignedat}
\end{equation}
But, as already explained above, $\|\eta u\|_{\bL_{p,\Theta}(G,T)}$ is indeed finite since the right hand side of~\eqref{eqn 4.9.3} is assumed to be finite.
Therefore, Estimate~\eqref{eqn 4.9.3.2} holds.
Moreover, it proves~\eqref{eqn 4.9.3}, since \icp due to \eqref{rho} and that fact that  $\eta=1$ on $U_1$,
\begin{equation*}
\begin{alignedat}{1}
\sum_{|\alpha|\leq m+1}\int_{U_1}  \Abs{\rho^{\abs{\alpha}-1}D^\alpha u}^p\rho^{\Theta-2}\,dx 
&\leq 
\sum_{|\alpha|\leq m+1}\int_{G}  \Abs{\rho^{\abs{\alpha}-1}D^\alpha (u\eta^{2m+3})}^p\rho^{\Theta-2}\,dx \\
&\leq N\sum_{|\alpha|\leq m+1}\int_{G}  \Abs{\rho^{\abs{\alpha}-1}_GD^\alpha (u \eta^{2m+3})}^p\rho^{\Theta-2}_G \,dx\\
&=
N\nnrm{\eta^{2m+3} u}{H^{m+1}_{p,\Theta-p}(G)}^p,
\end{alignedat}
\end{equation*}
and, as already mentioned above, the right hand side of~\eqref{eqn 4.9.3.2} can be estimated from above by the right hand side of~\eqref{eqn 4.9.3}.
Note that all constants in the estimates above are independent of $T$.
\end{proof}

\vspace{0.1cm}

Now we can prove Theorem \ref{lem:estim:2DCone:Theta} by applying Lemma~\ref{lem:estim:2DCone:U1} to $u_n(t,x):=u(2^{2n}t,2^nx)$ for each $n\in\bZ$ and summing up the resulting estimates.

\begin{proof}[Proof of Theorem \ref{lem:estim:2DCone:Theta}]
For every $n\in\bZ$, let $u_n(t,x):=u(2^{2n}t, 2^nx)$, $x\in \cD$, $t\leq 2^{-2n}T$. Since
$u\in\cK^1_{p,\theta,0}(\cD,T)$ solves Equation~\eqref{eq:SHE:Intro} on $\cD$, for every $n\in \bZ$, $u_n\in\cK^1_{p,\theta,0}(\cD,2^{-2n}T)$ and
\begin{align*}
du_n
= \grklam{\Delta u_n +2^{2n}f^0_n+2^n(f^i_{n})_{x^i} }\,dt + 2^n g^k_n\, d(2^{-n}w^k_{2^{2n}t}), \qquad  t\in (0, 2^{-2n}T],
\end{align*}
on $\cD$ in the sense of distributions with $u_n(0,\cdot)=0$ and
\[
f^0_n(t,x)=f^0(2^{2n}t,2^nx), \,\, f^i_n(t,x)=f^i(2^{2n}t,2^nx),\,\, \text{ and }\,\,g_n(t,x)=g(2^{2n}t,2^nx).
\]
Note that $(2^{-n}w^k_{2^{2n}t})_{t\geq 0}$, $k=1,2,\ldots$, is a sequence of independent one-dimensional Wiener processes.
By Lemma \ref{lem:estim:2DCone:U1} applied to $u_n$ for  $t\leq 2^{-2n}T$, we have
\begin{align*}
\E &\int^{2^{-2n}T}_0 \;\sum_{k=0}^{m+1} \int_{U_1} |\rho^{k-1}(x)2^{nk}D^{k}u(2^{2n}t,2^nx)|^p \rho^{\Theta-2}(x)\,dx \,dt\\
&\leq N \;\E \int^{2^{-2n}T}_0 \int_{V_1}\ssgrklam{ \; |u(2^{2n}t,2^nx)|^p  \\
&  \hspace{3.4cm}+\sum_{k=0}^{(m-1)\vee 0}|\rho^{k+1}(x)2^{2n}2^{nk}D^kf^0(2^{2n}t,2^nx)|^p \\
&\hspace{3.4cm}+\sum_{k=0}^m \sum_i|\rho^k(x) 2^n 2^{nk} D^kf^i(2^{2n}t,2^nx)|^p \\
&\hspace{3.4cm}+\sum_{k=0}^m|\rho^k(x) 2^n 2^{nk} D^kg(2^{2n}t,2^nx)|^p_{\ell_2}\;}\rho^{\Theta-2}(x)\,     dx\, dt.
\end{align*}
Thus, if for $n\in\bZ$,
$$
U_n:=\{x\in \cD: 2^{n-1}<|x|<2^{n+1} \}\quad\text{and}\quad V_n:=\{x\in \cD: 2^{n-2}<|x|<2^{n+2}\},
$$
then, multiplying both sides by $2^{n(\theta-2-p)}$, changing variables $(2^{2n}t,2^n x)\to (t,x)$, and using the relations $\rho(2^{-n}x)=2^{-n}\rho(x)$ and $|x|\sim 2^n $ on $U_{n-1}$ and on $V_{n-1}$, we get
\begin{align*}
\E \int^T_0 &\sum_{k=0}^{m+1} \int_{U_{n-1}} |\rho^{k-1} D^{k}u|^p  \rho_\circ^{\theta-2}  \left(\frac{\rho}{\rho_{\circ}}\right)^{\Theta-2} \,dx \,dt\\
&\leq N\; \E \int^T_0 \int_{V_{n-1}}\ssgrklam{|\,\rho_{\circ}^{-1}u|^p+\sum_{k=0}^{(m-1)\vee 0}|\rho^{k+1}D^kf^0|^p \\
&\hspace{3.2cm} +\sum_{k=0}^m \sum_i|\rho^k D^kf^i|^p+ |\rho^k D^kg|^p_{\ell_2}} \rho_\circ^{\theta-2}  \ssgrklam{\frac{\rho}{\rho_{\circ}}}^{\Theta-2}  \, dx\, dt.
\end{align*}
By summing up with respect to $n\in \bZ$, we obtain the desired result. 
\end{proof}

The following uniqueness result on bounded $\cont^1$ domains will be used in Section~\ref{sec:Polygons} to treat the  stochastic heat equation on polygons.

\begin{lemma}\label{lem for uniqueness2}     
Let $G$ be a bounded $\cont^1$ domain in $\bR^d$. For $j=1,2$, let $p_j\geq 2$ and $\Theta_j\in (d-1,d-1+p_j)$, and assume that $u\in \frH^1_{p_1,\Theta_1,0}(G,T)$ is a solution to Equation~\eqref{eq:SHE:Intro} on $G$ with $f^0$, $f^i$ and $g$ satisfying
     $$f^0\in \bL_{p_j,\Theta_j+p_j}(G,T)\cap \bL_{p_j,d+p_j}(G,T),\quad f^i\in \bL_{p_j,\Theta_j}(G,T) \cap \bL_{p_j,d}(G,T),\;i=1,2, 
     $$
     $$
      g\in \bL_{p_j,\Theta_j}(G,T;\ell_2)\cap \bL_{p_j,d}(G,T;\ell_2),
     $$
for $j=1,2$.    
Then $u\in \frH^1_{p_2,\Theta_2,0}(G,T)$.
    
     \end{lemma}
     
     \begin{proof}
    By \cite[Theorem 2.9]{Kim2004}, we can define  $v_1$ and $v_2$ as the solution to the equation in $\frH^1_{p_1,d,0}(G,T)$ and $\frH^1_{p_2,d,0}(G,T)$ respectively. Denote $q_1=p_1 \vee p_2$ and $q_2=p_1 \wedge p_2$. Then by the uniqueness result in $\frH^1_{q_2,d,0}(G,T)$ we conclude $v_1=v_2$ and it belongs to $\frH^1_{q_1,d,0}(G,T)$ as $G$ is bounded. 
Also, due to Lemma~\ref{collection}(v) and the uniqueness in   $\frH^1_{p_1,\Theta_1 \vee d,0}(G,T)$, we conclude $v_1=u$. Now let $v_3$ be the solution to the problem in $\frH^1_{p_2,\Theta_2.0}(G,T)$.  The same argument as above shows $v_3=v_2$. 
\end{proof}

\mysection{Proof of  Lemma \ref{lem 4.5.1}}\label{4}

In this section we prove the second key auxiliary result of this article, Lemma~\ref{lem 4.5.1}. 
Throughout, we take and fix a $\cont^\infty$ radial function $\eta$ and a corresponding $\cont^1$ domain $G\subseteq \cD$  as in the proof of Lemma~\ref{lem:estim:2DCone:U1}.  Recall that $\eta(t)=1$ for $1\leq t\leq 4$.  As a consequence, there exists a constant $c>0$ such that 
\begin{equation}\label{eqn 5.6.544}
\sum_{n=-\infty}^{\infty}\eta(e^{n+t})>c>0, \quad \forall \, t\in \bR.
\end{equation}

Our proof of Lemma~\ref{lem 4.5.1} relies on the the following characterization of the $L_{p,\theta}^{[\circ]}(\cD)$-norm.

\begin{lemma} 
\label{lem 3.1}
Let  $p>1$ and $\theta\in \bR$.  Let $u\colon\cD\to\bR$ be a measurable function. 
\begin{enumerate}[leftmargin=*,label=\textup{(\roman*)}, wide] 
\item\label{lem 3.1.1} If $\eta$ and $G$ are as above,  then
$$
\|u\|^p_{L^{[\circ]}_{p,\theta}(\cD)} \sim   \sum_{n \in \bZ} e^{n\theta} \|\eta(|x|) u(e^nx)\|^p_{L_p(\cD)}= \sum_{n \in \bZ} e^{n\theta} \|\eta(|x|) u(e^nx)\|^p_{L_p(G)}.
$$
\item\label{lem 3.1.2} For any  function $\xi  \in C^{\infty}_0((0,\infty))$ we have
$$
\sum_{n \in \bZ} e^{n\theta} \|\xi(|x|) u(e^nx)\|^p_{L_p(\cD)}\leq N(\xi, \eta, p,\theta) \|u\|^p_{L^{[\circ]}_{p,\theta}(\cD)}.
$$
\end{enumerate}
\end{lemma}

\begin{proof}
To see~\ref{lem 3.1.1}, it is enough to repeat the proof of \cite[Remark~1.3]{Kry1999c}. Indeed, by the change of variables $e^nx \to x$,
$$
\sum_{n\in \bZ} e^{n\theta}\|\eta(|x|)u(e^nx)\|^p_{L_p(\cD)}=\int_{\cD} \zeta(x)|u(x)|^pdx,
$$
where
$$
\zeta(x)=\sum_{n\in \bZ} e^{n(\theta-2)}\eta^p(e^{-n}|x|) \sim |x|^{\theta-2},
$$
 see \cite[Remark~1.3]{Kry1999c}.  Moreover, since $\mathrm{supp}\,\eta\cap\cD\subset G$, the equality in~\ref{lem 3.1.1} is also satisfied.  Part~\ref{lem 3.1.2}
holds since
 $$
\sum_{n\in \bZ} e^{n(\theta-2)}\xi^p(e^{-n}|x|) \leq N(\xi,\eta,\theta,p) |x|^{\theta-2};
$$
see \cite[Lemma~1.4]{Kry1999c} for details.
\end{proof}

In addition to Lemma~\ref{lem 3.1}, we also need the following counterpart of   Lemma~\ref{lem 4.5.1} for the stochastic heat equation on bounded $\cont^1$ domains.
In the proof, we are going to use the common abbreviations
\[
\bH^n_p(T):=L_p(\Omega_T, \cP;H^n_p(\bR^d)) \quad\text{and}\quad
\bL_p(T;\ell_2):=L_p(\Omega_T,\cP;L_p(\bR^d;\ell_2)),
\]
for $n\in\bZ$.

\begin{lemma}
\label{lem krylov}
Let $G$ be a bounded $\cont^1$ domain, $\Theta\in \bR$,  $p\geq 2$, and $u\in \frH^{1}_{p,\Theta,0}(G,T)$ with $du=fdt +g\,dw^k_t$. Then $u\in L_p(\Omega; \cont([0,T]; L_{p,\Theta}(G))$, and for any $c>0$,

\begin{eqnarray*}
\E \sup_{t\leq T} \|u(t,\cdot)\|^p_{L_{p,\Theta}(G)}\le N \Big( c \|u\|^p_{\bH^{1}_{p,\Theta-p}(G,T)} 
 + c^{-1}\|f\|^p_{\bH^{-1}_{p,\Theta+p}(G,T)}+\|g\|^p_{\bL_{p,\Theta}(G,T;\ell_2)} \Big),
 \end{eqnarray*}
where $N= N(d,p,\theta,G,  T)$. In particular, if  $f=f^0+f^i_{x^i}$, then the right hand side above is bounded by a constant multiple of 
\[
 c \|u\|^p_{\bH^1_{p,\Theta-p}(G,T)}+c^{-1}\|f^0\|^p_{\bL_{p,\Theta+p}(G,T)}+c^{-1}\|f^i\|^p_{\bL_{p,\Theta}(G,T)}
 +\|g\|^p_{\bL_{p,\Theta}(G,T;\ell_2)}.
\]
 \end{lemma}

\begin{proof}
Introduce a partition of unity  $\zeta_0, \zeta_1,\zeta_2,\cdots, \zeta_M$
 of $G$ such that  $\zeta_0\in \cont^{\infty}_c(G)$ and
 $\zeta_j\in \cont^{\infty}_c(B_{r}(x_j))$ $(j=1,2,\cdots,M)$, where
  $x_j\in \partial G$ and $r<r_0/K_0$.  
For any $\Theta\in \bR$, $m\in \bZ$, and  $v\in H^m_{p,\Theta}(G)$, since $\zeta_0$ has compact support in $G$, we can consider $\zeta_0v$ as a function defined on the entire space, so that by Lemma~\ref{collection}\ref{col:cptsupport},
\begin{equation}
         \label{eqn 2018-12}
\|\zeta_0 v\|_{H^m_{p,\Theta}(G)} \sim \|\zeta_0v\|_{H^m_p(\bR^d)}.
\end{equation}
Also, by~\eqref{eqn 12.7.1}, for $j\geq 1$ and  any $\Theta, \nu\in \bR$,
\begin{equation}
\label{relation}
\|\psi^{\nu} \zeta_j v\|^p_{H^{m}_{p,\Theta}(G)} \sim \|(x^1)^{\nu}(\zeta_j v)(\Psi^{-1}_j)\|^p_{H^{m}_{p,\Theta}(\bR^d_+)},
\end{equation}
where $\Psi_j$ is the corresponding mapping from Definition~\ref{definition domain} related to $x_j\in \partial G$.
Thus
\begin{equation}\label{eqv} 
\begin{alignedat}{1}
\|v\|^p_{H^{m}_{p,\Theta}(G)}
&= \sgnnrm{\sum_{j=0}^M \zeta_j v}{{H^{m}_{p,\Theta}(G)}}^p
\sim \sum_{j=0}^M \|\zeta_j v\|^p_{H^{m}_{p,\Theta}(G)}\\
& \sim
\|\zeta_0 v\|^p_{H^{m}_p(\bR^d)} + \sum_{j= 1}^M \|(\zeta_j v)(\Psi^{-1})\|^p_{H^{m}_{p,\Theta}(\bR^d_+)}.  
\end{alignedat}
\end{equation}
As a consequence, 
\begin{equation}
  \nonumber \label{eqn 2018-15}
\E \sup_{t\leq T} \|u(t,\cdot)\|^p_{L_{p,\Theta}(G)} 
\leq N 
\ssgrklam{\E\sup_{t\leq T} \|\zeta_0u\|^p_{L_p(\bR^d)} +\sum_{j=1}^M \E \sup_{t\leq T} \|(\zeta_ju)(\Psi^{-1}_j)\|^p_{L_{p,\Theta}(\bR^d_+)}}.
\end{equation}
Therefore, in order to obtain the desired estimate and the continuity assertion, it is enough to estimate the terms on the right hand side appropriately and to prove continuity of $\zeta_ju$, $j=0,\ldots,M$. 
For the first term, note that 
$$
d(\zeta_0 u)=\zeta_0f \,dt+\zeta_0g\,dw^k_t, \quad  \,t\in (0,T],
$$
on $\bR^d$ and $\zeta_0u\in \bH^1_p(T)$ due to~\eqref{eqn 2018-12}.
Therefore, by \cite[Corollary~4.12]{Kry2001} (for $p>2$) and \cite[Remark~4.14]{Kry2001} (for $p=2$), 
\begin{equation}\label{eq:cts:zeta0}
\zeta_0 u \in L_p(\Omega; \cont([0,T]; L_p(\bR^d)),
\end{equation}
and there exists a constant $N$, such that, for any $c>0$,
\begin{align*}
\E \sup_{t\leq T} \|\zeta_0u\|^p_{L_p}
&\leq N c \|\zeta_0u\|^p_{\bH^1_p(T)}+
Nc^{-1}\|\zeta_0 f\|^p_{\bH^{-1}_p(T)}+N \|\zeta_0g\|^p_{\bL_p(T;\ell_2)}\\
&\leq Nc \|\zeta_0u\|^p_{\bH^1_{p,\Theta-p}(G,T)}+Nc^{-1}\|\zeta_0f\|^p_{\bH^{-1}_{p,\Theta}(G,T)}
+N\|\zeta_0g\|^p_{\bL_{p,\Theta}(G,T;\ell_2)}.
\end{align*}
Moreover, for every $j\in\{1,\ldots,M\}$,
\[
d((\zeta_j u)(\Psi^{-1}_j))
=
(\zeta_j f)(\Psi^{-1}_j)\,dt+(\zeta_jg^k)(\Psi^{-1}_j) \,dw^k_t =:F_j \,dt+G^k_j \,dw^k_t, \quad t\in (0,T],
\]
on $\bR^d_+$ and due to~\eqref{relation}, $(\zeta_ju)(\Psi^{-1}_j) \in \frH^1_{p,\Theta}(\bR^d_+,T)$ (see  Section~\ref{sec:proof:lift}  for  notation). Therefore,
by \cite[Theorem~4.1]{Kry2001} (for $p>2$) and \cite[Remark~4.5]{Kry2001} (for $p=2$),
\begin{equation}\label{eq:cts:zetaj}
(\zeta_ju)(\Psi^{-1}_j)\in L_p(\Omega; \cont([0,T]; L_{p,\Theta}(\bR^d_+)),
\end{equation}
and
\begin{align*}
\E &\sup_{t\leq T} \|(\zeta_ju)(\Psi^{-1}_j)\|^p_{L_{p,\Theta}(\bR^d_+)}\\
&\leq 
N c \|(\zeta_ju)(\Psi^{-1}_j)\|^p_{\bH^1_{p,\Theta-p}(\bR^d_+,T)}
+Nc^{-1}\|F_j\|^p_{\bH^{-1}_{p,\Theta+p}(\bR^d_+,T)}+N\|G_j\|^p_{\bL_{p,\Theta}(\bR^d_+,T;\ell_2)}\\
&\leq N c \|\zeta_ju\|^p_{\bH^1_{p,\Theta-p}(G,T)}
+Nc^{-1}\|\zeta_jf\|^p_{\bH^{-1}_{p,\Theta+p}(G,T)}+
N\|\zeta_jg\|^p_{\bL_{p,\Theta}(G,T;\ell_2)}.
\end{align*}
Summing up gives the desired estimate and $u\in L_p(\Omega;\cont([0,T];L_{p,\Theta}(G)))$ follows from~\eqref{eq:cts:zeta0} and \eqref{eq:cts:zetaj}, together with~\eqref{eqv}.
The second assertion is due to the fact that, by Lemma~\ref{collection}, 
\[
\|f^i_{x^i}\|_{H^{-1}_{p,\Theta+p}(G)}\leq N \|\psi f^i_{x^i}\|_{H^{-1}_{p,\Theta}(G)}\leq N \|f^i\|_{L_{p,\Theta}(G)}.\qedhere
\]
\end{proof}

We have now all ingredients we need in order to prove Lemma~\ref{lem 4.5.1}.

\begin{proof}[Proof of  Lemma~\ref{lem 4.5.1}]
We first prove Estimate~\eqref{eq:estim:sup:2DCone}.  By Lemma~\ref{lem 3.1},
\begin{equation}\label{eqn 2018-4}
\E \sup_{t\leq T} \|u(t,\cdot)\|^p_{L^{[\circ]}_{p,\theta}(\cD)} 
\leq 
N \sum_{n\in \bZ} e^{n\theta} \,\E \sup_{t\leq T} \|u(t,e^nx) \eta(x)\|^p_{L_p(G)}.
\end{equation}
 For $n\in\bZ$, let  $v_n(t,x):=u(t,e^nx)\eta(x)$. Then 
$$
dv_n=[ e^{-n} (f^i(t,e^nx))_{x^i}\eta(x)+f^0(t,e^nx)\eta(x)]dt+g^k(t,e^nx)\eta(x)dw^k_t, \quad t\in(0,T],
$$
on $G$.
Note that
$$
 e^{-n} (f^i(t,e^nx))_{x^i}\eta(x)=  e^{-n} [f^i(t,e^nx)\eta(x)]_{x^i}-e^{-n}f^i(t,e^nx)\eta_{x^i}(x),
$$
and
\begin{equation}\label{eq_sec4_2}
(v_n)_{x^i}=e^{n}u_{x^i}(t,e^nx)\eta(x)-u(t,e^nx)\eta_{x^i}(x).
\end{equation}
Obviously, $L_{p,2}(G)=L_p(G)$  and  by Hardy's inequality,
\begin{equation}\label{eq_sec4_1}
\|v_n\|_{H^1_{p,2-p}(G)}
\le 
N\sgrklam{\|\rho^{-1}_G v_n\|_{L_p(G)}+\sum_i\|(v_{n})_{x^i}\|_{L_p(G)} }
\leq 
N \|(v_{n})_{x}\|_{L_p(G)}.
\end{equation}
By Lemma~\ref{lem krylov} with $\Theta=d=2$, \eqref{eq_sec4_1}, and \eqref{eq_sec4_2}, for any $c>0$
\begin{align*}
&\E \sup_{t\leq T} \|v_n(t,\cdot)\|^p_{L_p(G)} \\
&\leq N \Big(ce^{np}\sum_i\|u_{x^i}(\cdot,e^n\cdot)\eta\|^p_{\bL_{p,d}(G,T)}+
c\sum_i\|u(\cdot,e^n\cdot)\eta_{x^i}\|^p_{\bL_{p,d}(G,T)} \\
\\&\quad \quad +e^{-np}c^{-1}\|f^i(\cdot,e^n\cdot)\eta\|^p_{\bL_{p,d}(G,T)}
+e^{-np}c^{-1}\sum_i\| \rho f^i(\cdot,e^n\cdot)\eta_{x^i}\|^p_{\bL_{p,d}(G,T)}\\
&\quad \quad +c^{-1}\|\rho f^0(\cdot,e^n\cdot)\eta\|^p_{\bL_{p,d}(G,T)} +\|\eta g(\cdot,e^n\cdot)\|^p_{\bL_{p,d}(G,T;\ell_2)} \Big).
\end{align*}
Since $\rho$ is bounded in $G$, we can drop $\rho$ above, so that, if we choose
$c:=e^{-np}$ and use~\eqref{eqn 2018-4}, we get
\begin{align*}
\E &\sup_{t\leq T} \|u(t,\cdot)\|^p_{L^{[\circ]}_{p,\theta}(\cD)} \\
&\leq N 
\ssgrklam{\sum_n e^{n\theta}\|u_x(\cdot,e^n\cdot)\eta\|^p_{\bL_{p,d}(\cD,T)}
+ \sum_n e^{n(\theta-p)}\sum_i\|u(\cdot,e^n\cdot)\eta_{x^i}\|^p_{\bL_{p,d}(\cD,T)}\\
&\quad\quad\quad+ \sum_{n,i} e^{n\theta}\|f^i(\cdot,e^n\cdot)\eta\|^p_{\bL_{p,d}(\cD,T)}
+ \sum_{n,i} e^{n\theta} \|f^i(\cdot,e^n\cdot)\eta_{x^i}\|^p_{\bL_p(\cD,T)}\\
&\quad\quad\quad+  \sum_n e^{n(\theta+p)}\|f^0(\cdot,e^n\cdot)\eta\|^p_{\bL_{p,d}(\cD,T)}
+ N \sum_n e^{n\theta}\|g(\cdot,e^n\cdot)\eta\|^p_{\bL_{p,d}(\cD,T;\ell_2)}}.
\end{align*}
Therefore,  due to Lemma~\ref{lem 3.1}\ref{lem 3.1.2}, Estimate~\eqref{eq:estim:sup:2DCone} holds.

To prove the continuity assertion, we take a sequence of smooth functions $\xi_n\in \cont^{\infty}_c(\bR^2)$ such that $\xi_n=1$ if $3/n<|x|<n$, $\xi_n(x)=0$ if $|x|\leq 2/n$ or $|x|\geq 2n$, and 
\begin{equation}
   \label{eqn 4.9.9}
\sup_i\sup_n \sup_x |x|\, |(\xi_n)_{x^i}| < \infty.
\end{equation}
Moreover, for every $n\in\bN$, let $G_n\subset\cD$ be a bounded $\cont^1$ domain such that  
$$
\cD \cap \{x: 2/n <|x|<2n\} \subset G_n \subset \cD \cap \{x: 1/n <|x|<3n\}.
$$
Then
$$
d(\xi_n u)=\left[ (\xi_n f^i)_{x^i}-(\xi_n)_{x^i} f^i+\xi_nf^0\right]dt+\xi_n g^kdw^k_t, \quad t\in (0,T],
$$
on $G_n$ and,  by the choice of $\xi_n$, we also know that $\xi_n u \in \frH^1_{p,2}(G_n,T)$ since $\xi_n u\in\cK^1_{p,\theta,0}(\cD,T)\subseteq \mathring{\wso}^1_{p,\theta-p}(\cD,T)$, see also~\cite[Remark~3.2]{CioKimLee+2018}.  By Lemma~\ref{lem krylov} with $\Theta=2$, $\xi_n u \in L_p(\Omega; \cont([0,T]; L_p(G_n))$.
  Since $\xi_n$ vanishes near the origin and toward infinity, we conclude  
$$
\xi_n u \in L_p(\Omega; \cont([0,T]; L_{p,\Theta}(\cD)).
$$
Applying~\eqref{eq:estim:sup:2DCone} to $\xi_n u -\xi_m u$ and using~\eqref{eqn 4.9.9}, we find that $\xi_n u$ is a Cauchy sequence in  
$L_p(\Omega; \cont([0,T]; L_{p,\Theta}(\cD))$, which converges in this space to a limit $v$. Moreover, applying~\eqref{eq:estim:sup:2DCone} to $\xi_n u-u$,  we find that $u=v$ in $L_p(\Omega; L_{\infty}([0,T]; L_{p,\Theta}(\cD))$. Therefore, $u(t)=v(t)$ for all $t\leq T$ ($\prob$-a.s.).  Thus $u$ has the desired version with continuous paths. 
\end{proof}

\mysection{The stochastic heat equation on polygons}\label{sec:Polygons}

 In this section we present our analysis for the stochastic heat equation~\eqref{eq:SHE:Intro} on polygons in $\bR^2$. We fix the following setting: Throughout, 
let $\cO \subset \bR^2$ be a polygon with vertexes $\{v_1,v_2, \ldots, v_M\}\subset\bR^2$.  For $x\in \cO$, put
$$
\tilde{\rho}(x):=\min_{1\leq j\leq M} |x-v_j|, \quad  \quad \rho(x):=\rho_\domain(x):=\text{dist}(x,\partial \cO),
$$
and for $j=1,\ldots,M$ define
$$
\kappa_j:=\text{interior angle at}\, v_j, \quad \quad \kappa_0:=\max_{1\leq j\leq M}\kappa_j.
$$

Motivated by the analysis of the stochastic heat equation on angular domains from Section~\ref{sec:2DCone}, we are going to use weighted Sobolev spaces with weights based on the distance $\tilde{\rho}$ to the set of vertexes in order to establish existence and uniqueness of a solution. More precisely,  for $\theta\in \bR$, $p>1$ and $n\in \{0,1,2,\ldots\}$,  we define the spaces  $K^n_{p,\theta}(\cO)$, $K^n_{p,\theta}(\cO;\ell_2)$,   $L^{[\circ]}_{p,\theta}(\cO)$, and 
$L^{[\circ]}_{p,\theta}(\cO;\ell_2)$  in the same way as the corresponding spaces on $\cD$ from Section~\ref{sec:2DCone} with $\rho_\circ$ replaced by $\tilde{\rho}$, i.e., for instance,
$$
\|u\|^p_{K^n_{p,\theta}(\cO)}=\sum_{|\alpha|\leq n} \int_{\cO} |\tilde{\rho}^{|\alpha|} D^{\alpha}u|^p \tilde{\rho}^{\theta-2}dx.$$
The space $\mathring{\wso}^1_{p,\theta}(\cO)$ is  the closure of the space $\cont^\infty_c(\cO)$ of test functions in $\wso^1_{p,\theta}(\cO)$. 
In analogy to Section~\ref{sec:2DCone}, for the $L_p$-spaces of predictable stochastic processes with values in the weighted Sobolev spaces introduced above we use the  abbreviations 
\[
\bwso^{n}_{p,\theta}(\cO,T)
:=
L_p(\Omega_T, \pred;\wso^{n}_{p,\theta}(\cO)),
\qquad \bwso^{n}_{p,\theta}(\cO,T;\ell_2)
:=
L_p(\Omega_T, \pred;\wso^{n}_{p,\theta}(\cO;\ell_2)),
\]
\[
\bL^{[\circ]}_{p,\theta}(\cO,T)
:=\bwso^0_{p,\theta}(\cO,T), \quad\quad\quad  
\bL^{[\circ]}_{p,\theta}(\cO,T,\ell_2)
:=\bwso^0_{p,\theta}(\cO,T;\ell_2),
\]
and
\[
\mathring{\bwso}^1_{p,\theta}(\cO,T)
:=
L_p(\Omega_T, \pred ;\mathring{\wso}^{1}_{p,\theta}(\cO)).
\]
Moreover, $\cK^1_{p,\theta,0}(\domain,T)$ is defined the following way.
\begin{defn}
Let $p \geq 2$. We write $u\in\cK^1_{p,\theta,0}(\cO,T)$  if
$u\in\mathring{\bwso}^1_{p,\theta-p}(\cO,T)$
 and
there exist $f^0\in \bL^{[\circ]}_{p,\theta+p}(\cO,T)$, $f^i\in  \bL^{[\circ]}_{p,\theta}(\cO,T)$, $i=1,2$, and
 $g\in \bL^{[\circ]}_{p,\theta}(\cO,T;\ell_2)$
such that 
\begin{equation*}\label{eqn 28_2}
 du=(f^0+f^i_{x^i})\, dt +g^k \, dw^k_t,\quad t\in(0,T],
\end{equation*}
on $\domain$ in the sense of distributions with $u(0,\cdot)=0$; see Definition~\ref{defn sol} accordingly. In this situation
we also write
$$
\bD u:=f^0+f^i_{x^i}\qquad\text{and}\qquad \bS u :=g.
$$
\end{defn}

In this article, Equation~\eqref{eq:SHE:Intro} has the following meaning on $\domain$.

\begin{defn}
We say that  $u$ is a solution to Equation~\eqref{eq:SHE:Intro} on $\cO$ in the class $\mathcal{K}^{1}_{p,\theta,0}(\cO,T)$ 
if
$u\in \mathcal{K}^{1}_{p,\theta,0}(\cO,T)$ with
\[
\bD u = \Delta u + f^0+f^i_{x^i}=f^0+(f^i+u_{x^i})_{x^i}
\qquad
\text{and}
\qquad
\bS u = g.
\]
\end{defn}

Before we look at Equation~\eqref{eq:SHE:Intro} in detail, we first prove the following version of Lemma~\ref{lem 4.5.1} for polygons.
It is a key ingredient in our existence and uniqueness proof below.

\begin{lemma}
\label{lem for gronwall}
Let $p\geq 2$ and $\theta \in \bR$. Assume that $u\in\cK^1_{p,\theta,0}(\cO,T)$, such that $du=(f^0+f^i_{x^i})dt+g^kdw^k_t$ with
\[
f^0\in \bL^{[\circ]}_{p,\theta+p}(\cO,T), 
\quad f^i \in \bL^{[\circ]}_{p,\theta}(\cO,T),\,i=1,2, \quad\text{and } g\in \bL^{[\circ]}_{p,\theta}(\cO,T;\ell_2).
\]
 Then $u\in L_p(\Omega; \cont([0,T];L^{[\circ]}_{p,\theta}(\cO)))$ and 
 \begin{align*}
\E \sup_{t\leq T}& \|u(t,\cdot)\|^p_{L^{[\circ]}_{p,\theta}(\cO)} \\
&\leq N \Big(\|u\|^p_{\bwso^{1}_{p,\theta-p}(\cO,T)}
+\|f^0\|^p_{\bL^{[\circ]}_{p,\theta+p}(\cO, T)}+\sum_i\|f^i\|^p_{\bL^{[\circ]}_{p,\theta}(\cO, T)}+\|g\|^p_{\bL^{[\circ]}_{p,\theta}(\cO, T; \ell_2)} \Big)\\
&=:N \,C(u,f^0,f^i,g, T),
 \end{align*}
where $N=N(d,p,\theta,T)$ is a non-decreasing function of $T$.
In particular, for any $t\leq T$,
\[
\|u\|^p_{\bL^{[\circ]}_{p,\theta}(\cO,t)}
\leq 
\int^t_0 \E\sup_{r\leq s} \|u(r)\|^p_{L^{[\circ]}_{p,\theta}(\cD)} ds\leq  N(d,p, \theta,  T) \int^t_0 C(u,f^0,f^i,g, s) \,ds.
 \]
 \end{lemma}

\begin{proof}
We combine Lemma \ref {lem 4.5.1} (see also Remark~\ref{remark cones}) and Lemma~\ref{lem krylov} as follows.

Fix a sufficiently small $r>0$ such that  $B_{3r}(v_j)$ contains only one vertex $v_j$  and intersects with only two edges   for each $j\leq M$. 
Choose a function $\xi\in \cont^{\infty}_c(\bR^2)$ such that $0<\xi(x)\leq 1$ for $|x|<2r$, $\xi(x)=1$ for $|x|<r$, and $\xi(x)=0$ for $|x|\geq 2r$. 
Let $\xi_j(x):=\xi(x-v_j)$ and put $\xi_0(x):=1-\sum_{j=1}^M \xi_j(x)$, $x\in\domain$.  
Note that by the choice of $r$ and $\xi$, the supports of the $\xi_j$'s ($j\geq 1$) do not overlap, and therefore $0\leq \sum_{j=1}^M \xi^j \leq 1$. Moreover, $\xi_0(x)=1$ if  $x$ is not close to vertexes, that is,  if $\tilde{\rho}(x)=\min_{1\leq j\leq M} |x-v_j|\geq r$.
 For $j=1,2,\ldots,M$, let $\cD_j$ be the angular domain centered at $v_j$ with interior angle $\kappa_j$ such that 
 \begin{equation}
    \label{eqn 4.10.10}
 \cD_j \cap B_{3r}(v_j)=\cO \cap B_{3r}(v_j).
 \end{equation}
 Moreover, let $G$ be a $\cont^1$ domain in $\cO$ such that 
 \begin{equation}
     \label{eqn 4.10.12}
 \xi_0(x)=0 \quad \text{for }\, x\in \cO\setminus G \quad \text{and}\quad \inf_{x\in G} \tilde{\rho}(x)>c>0.
 \end{equation}
 Note that by the choice of $\xi_j$, $\cD_j$, $j=0,1,\ldots,M$, and $G$, for any $\theta\in\bR$ \icp  and $v\in \mathring{K}\icp^1_{p,\theta-p}(\cO)$, 
\begin{equation}
 \label{eqn 4.11.1}
\|\xi_0 v\|_{L^{[\circ]}_{p,\theta}(\cO)} 
\sim 
\|\xi_0 v\|_{L_p(G)}, \quad \quad   
\|\xi_0v\|_{K^1_{p,\theta-p}(\cO)}\sim \|\xi_0v\|_{H^1_{p,2-p}(G)},
\end{equation}
\begin{equation}
 \label{eqn 4.11.2}
\|\xi_j v\|_{K^1_{p,\theta-p}(\cO)}=\|\xi_j v\|_{K^1_{p,\theta-p}(\cD_j)} \quad (j\geq 1).
 \end{equation}
The first and the third relation are trivial  and hold actually for arbitrary measurable $v\colon\domain\to \bR$, provided the expressions make sense. The second one is due  to \eqref{eqn 4.10.12} and  Hardy's inequality as
 \begin{align*}
\|\xi_0v\|_{K^1_{p,\theta-p}(\cO)}
&\leq 
N (\|\xi_0v\|_{L_p(G)}+\sum_i\|(\xi_0v)_{x^i}\|_{L_p(G)}) \\
& \leq  N \|\xi_0 v\|_{H^1_{p,2-p}(G)}\leq N \sum_i\|(\xi_0v)_{x^i}\|_{L_p(G)} 
\leq N \|\xi_0 v\|_{K^1_{p,\theta-p}(\cO)}.
 \end{align*}
 The three relations from~\eqref{eqn 4.11.1} and~\eqref{eqn 4.11.2} together imply, in particular, that 
\[
\|v\|_{L^{[\circ]}_{p,\theta}(\cO)}
\sim 
\sum_{j=0}^M \|\xi_j v\|_{L^{[\circ]}_{p,\theta}(\cO)}   
\sim 
\|\xi_0v\|_{L_p(G)}+ \sum_{j=1}^M \|\xi_j v\|_{L^{[\circ]}_{p,\theta}(\cD_j)},
\]
\begin{eqnarray}
 \label{eqn 4.11.4}
\|v\|_{K^1_{p,\theta-p}(\cO)} 
\sim 
\sum_{j=0}^M \|\xi_j v\|_{K^1_{p,\theta-p}(\cO)}
\sim
\|\xi_0v\|_{H^1_{p,2-p}(G)}
+
\sum_{j=1}^M \|\xi_j v\|_{K^1_{p,\theta-p}(\cD_j)}.
\end{eqnarray}
Also note that for  any multi-index $\alpha$,
\begin{equation}
   \label{eqn 4.10.14}
\sum_{j=0}^M \|v D^{\alpha}\xi_j\|_{L^{[\circ]}_{p,\theta}(\cO)}
+  
\sum_{j=0}^M \|v \tilde{\rho}D^{\alpha}\xi_j\|_{L^{[\circ]}_{p,\theta}(\cO)}
\leq N 
\|v\|_{L^{[\circ]}_{p,\theta}(\cO)}.
\end{equation}

Using the preparations above, we can verify the assertion the following way. 
For each $j\in\{1,2,\ldots,M\}$, $u^j:=\xi_j u \in \cK^1_{p,\theta,0}(\cD_j,T)$ with
\begin{equation}
     \label{eqn 4.10.11}
  du^j=\big((\xi_jf^i)_{x^i}+\xi_jf^0-(\xi_j)_{x^i}f^i\big)\,dt+ \xi_jg^k\, dw^k_t, \quad t\in(0,T],
  \end{equation}
on $\cD_j$ in the sense of distributions.
Thus, by Lemma~\ref{lem 4.5.1} (see also Remark~\ref{remark cones}) and (\ref{eqn 4.10.10}),
  $ u^j \in L_p(\Omega; \cont([0,T];L^{[\circ]}_{p,\theta}(\cO)))$, and 
\begin{align}
\E \sup_{t\leq T} \|\xi_j u\|^p_{L^{[\circ]}_{p,\theta}(\cO)}
\leq &N\Big( \|\xi_j f^0\|^p_{\bL^{[\circ]}_{p,\theta+p}(\cO,T)}+\sum_i\|\xi_j f^i\|^p_{\bL^{[\circ]}_{p,\theta}(\cO,T)}\label{eqn 4.10.13}\\
  &+ \|\xi_ju\|^p_{\bwso^{1}_{p,\theta-p}(\cO,T)}+\|(\xi_j)_{x^i}\tilde{\rho}f^i \|^p_{\bL^{[\circ]}_{p,\theta}(\cO,T)}  +
  \|\xi_j g\|^p_{\bL^{[\circ]}_{p,\theta}(\cO,T;\ell_2)} \Big). \nonumber
  \end{align}
Also, $u^0:=\xi_0 u\in \frH^1_{p,2,0}(G,T)$ and \eqref{eqn 4.10.11} holds with $j=0$. Thus, by Lemma~\ref{lem krylov} and \eqref{eqn 4.10.12}, $u^0\in  L_p(\Omega; \cont([0,T];L^{[\circ]}_{p,\theta}(\cO)))$, and \eqref{eqn 4.10.13} holds with $j=0$. Therefore, by summing up all these estimates and using above relations, we get the desired result. 
\end{proof}

Our main existence and uniqueness result for the stochastic heat equation on polygons reads as follows.
Recall that in this section $\kappa_0$ denotes the maximum over all  interior angles of the polygon $\cO$.

\begin{thm}[Existence and uniqueness/polygons]
    \label{thm polygon main}
Let $p\geq 2$ and assume that $\theta\in\bR$ satisfies \eqref{eq:range:vertex}. Then for any 
\[
f^0\in \bL^{[\circ]}_{p,\theta+p}(\cO,T), \quad
f^i\in \bL^{[\circ]}_{p,\theta}(\cO,T),\,i=1,2,\quad\text{and}\quad 
g  \in \bL^{[\circ]}_{p,\theta}(\cO,T;\ell_2),
\]
Equation~\eqref{eq:SHE:Intro} on $\cO$ 
has a unique solution $u\in\cK^1_{p,\theta,0}(\cO,T)$. 
Moreover, 
\begin{equation}
 \label{eqn polygon main}
\|u\|_{\bwso^{1}_{p,\theta-p}(\cO,T)}
\leq 
N \,\sgrklam{\|f^0\|_{ \bL^{[\circ]}_{p,\theta+p}(\cO,T)}
+
\sum_i\|f^i\|_ {\bL^{[\circ]}_{p,\theta}(\cO,T)}
+
\|g\|_{ \bL^{[\circ]}_{p,\theta}(\cO,T;\ell_2)}},
\end{equation}
where $N=N(p,\theta,\kappa_0,T)$.
\begin{proof}
\emph{Step 1.}
We first prove that~\eqref{eqn polygon main} holds given that a solution $u\in\cK^1_{p,\theta,0}(\cO,T)$ already  exists, by using corresponding results for the stochastic heat equation on angular domains and on $\cont^1$ domains. This will, in particular, take care of the uniqueness.

 Let $r>0$, $\xi_j$, $\cD_j$, $j=1,\ldots,M$, as well as $\xi_0$ and $G$ be as in the proof of Lemma~\ref{lem for gronwall}. 
A very similar reasoning as therein can be used to verify that $\xi_0u\in\frH^{1}_{p,\theta,0}(G,T)$, $\xi_ju\in\cK^1_{p,\theta,0}(\cD_j,T)$ for $j\geq 1$, and that for all $j\in\{0,1,\ldots,M\}$, 
\begin{align*}
d(\xi_j u)
&=
\grklam{\Delta (\xi_j u)+(-2u(\xi_j)_{x^i}+\xi_j f^i )_{x^i}\\
&\qquad\qquad\qquad\qquad+u\Delta \xi_j-(\xi_j)_{x^i}f^i+\xi_jf^0}\,dt+ \xi_j g^k \,dw^k_t, \quad t\in (0,T].
\end{align*}
Thus, by Theorem~\ref{thm:ex:uni:2DCone} for $j\geq 1$ and by Lemma~\ref{lem 10} for $j=0$ (see also \cite[Theorem~2.9]{Kim2004}), we obtain the estimate for 
$\|\xi_j u\|^p_{\bwso^{1}_{p,\theta-p}(\cO,t)}$ for each $t\leq T$. Then summing up over all $j$ and using \eqref{eqn 4.11.4} and \eqref{eqn 4.10.14}, yields that for each $t\leq T$,
\begin{align}
\|u\|_{\bwso^{1}_{p,\theta-p}(\cO,t)}
&\leq N \sgrklam{
\|u\|^p_{\bL^{[\circ]}_{p,\theta}(\cO,t)}
+\label{2018-7}
\|f^0\|^p_{\bL^{[\circ]}_{p,\theta+p}(\cO,T)}\\
&\qquad\qquad\qquad\qquad +
\sum_i\|f^i\|^p_{\bL^{[\circ]}_{p,\theta}(\cO,T)}
+
\|g\|^p_{\bL^{[\circ]}_{p,\theta}(\cO,T;\ell_2)} }. \nonumber 
\end{align}
%To apply Lemma \ref{lem for gronwall}, 
Recall that
$$
du=(f^0+(f^i+u_{x^i})_{x^i})\,dt+g^k\,dw^k_t, \quad t\leq T,
$$
and $\|u_x\|_{\bL^{[\circ]}_{p,\theta}(\cO,s)}\leq N \|u\|_{\bwso^1_{p,\theta-p}(\cO,s)}$. 
Thus, by Lemma~\ref{lem for gronwall} and \eqref{2018-7}, for each $t\leq T$,
\begin{align*} 
\|u\|^p_{\bwso^{1}_{p,\theta-p}(\cO,t)} 
&\leq N \int^t_0  \|u\|^p_{\bwso^{1}_{p,\theta-p}(\cO,s)} \,ds\\
&\qquad + N\sgrklam{\|f^0\|^p_{\bL^{[\circ]}_{p,\theta+p}(\cO,T)}+ \sum_i\|f^i\|^p_{\bL^{[\circ]}_{p,\theta}(\cO,T)}
  +  \|g\|^p_{\bL^{[\circ]}_{p,\theta}(\cO,T;\ell_2)}}.
\end{align*}
Hence the desired estimate follows by Gronwall's inequality.

\smallskip
 
\noindent\emph{Step 2.}  We prove existence as follows.
 Due to Lemma~\ref{lem for gronwall} 
and the a-priori estimate obtained in Step~1, we may assume $f^0$, $f^i$, $i=1,2$, and $g$ are very nice  in the sense that they vanish near the boundary and 
\[
f^i, f^i_{x^i}, f^0\in L_2(\Omega_T,\cP; L_2(\cO)), \quad \text{and}\quad g \in L_2(\Omega_T,\cP; L_2(\cO;\ell_2)).
\]
Then, by classical results (see, for instance, \cite{Roz1990} or \cite[Theorem~2.12]{Kim2014}), there exists a unique solution $u$ in $\frH^1_{2,2,0}(\cO,T)$, which satisfies, in particular,  
\begin{equation}
   \label{eqn 4.11.7}
\rho^{-1}u, u_{x^i} \in L_2(\Omega_T,\cP;L_2(\cO)),\;i=1,2, \;\text{ and }\; \sup_x |u|\leq N \|u_{x}\|_{L_2(\cO)}.
 \end{equation}
Note that for each $j \geq 1$,
\begin{equation}
\label{eqn 4.11.10}
d(\xi_j u)=\grklam{\Delta (\xi_j u)+f^{j,i}_{x^i}+f^{j,0} }\,dt+ \xi_j g^k \,dw^k_t, \quad t\in(0,T],
\end{equation}
on $\cD_j$, where, due to \eqref{eqn 4.11.7} and the fact that $(\xi_j)_{x^i}=0$ near the vertex $v_j$,
\begin{equation}
\label{fji}
f^{j,i}:=-2(\xi_j)_{x^i} u+\xi_jf^i \in \bL^{[\circ]}_{p,\theta}(\cD_j,T) \cap \bL^{[\circ]}_{2,2}(\cD_j,T),
\end{equation}
\begin{equation}
\label{fj0}
f^{j,0}:=u\Delta \xi_j+f^0\xi_j  + \sum_i (\xi_j)_{x_i}f^i  \in \bL^{[\circ]}_{p,\theta+p}(\cD_j,T) \cap \bL^{[\circ]}_{2,2+2}(\cD_j,T), 
\end{equation}
and 
\[
\xi_j g \in  \bL^{[\circ]}_{p,\theta}(\cD_j,T;\ell_2) \cap \bL^{[\circ]}_{2,2}(\cD_j,T;\ell_2).
\]
Since $\tilde{\rho}(x) \geq   \rho(x)$,  it follows that for each for $j\geq 1$ we have $\xi_ju\in \cK^1_{2,2,0}(\cO,T)$.  Thus, by  Lemma \ref{lem for uniqueness}, we conclude $\xi_ju\in \cK^1_{p,\theta,0}(\cO,T)$ if $j\geq 1$. Similar arguments based on Lemma~
\ref{lem for uniqueness2}  yield that $\xi_0 u\in \frH^1_{p,2,0}(G,T)$. Therefore, $\xi_0 u\in  \cK^1_{p,\theta,0}(\cO,T)$ (see~\eqref{eqn 4.11.1}), and consequently 
$u\in  \cK^1_{p,\theta,0}(\cO,T)$.
\end{proof}
\end{thm}

\begin{remark}
\label{remark 8.24}
 Note that even if we were to consider Equation~\eqref{eq:SHE:Intro} on $\domain$ with $f^i=0$, $i=1,2$, our proof strategy for Theorem~\ref{thm polygon main} (and Theorem~\ref{thm_polygons_1} below) requires that we are able to handle the localized equation on $\cD_j$ with forcing term $((\xi_j)_{x^i}u)_{x^i}$, which means that we have to be able to treat Equation~\eqref{eq:SHE:Intro} on angular domains with $f^i\neq 0$.  
This is why we need the extension of \cite[Theorem~3.7]{CioKimLee+2018} presented in Theorem~\ref{thm:ex:uni:2DCone} even for the proof of Theorem~\ref{thm polygon main} with $f^i=0$, $i=1,2$.
\end{remark}

We conclude with our main higher order regularity result for the stochastic heat equation on polygons.

\begin{thm}[Higher order regularity/polygons]\label{thm_polygons_1}
 Given the setting of Theorem~\ref{thm polygon main}, let $u$ be the unique solution in the class $\cK^{1}_{p,\theta,0}(\cO,T)$ to Equation~\eqref{eq:SHE:Intro} on $\cO$.  
Assume that 
\begin{align*}
C(m,\theta, f^i,f^0,g)
&:=
\E \int^T_0 \int_{\cO} \ssgrklam{\sum_{|\alpha|\leq (m-1)\vee 0\icp} |\rho^{\abs{\alpha}+1}D^{\alpha}f^0|^p+ \sum_i\sum_{|\alpha|\leq m} |\rho^{|\alpha|}D^{\alpha}f^i|^p\\
&\qquad\qquad\qquad\qquad+|\tilde{\rho}f^0|^p+
\sum_{\abs{\alpha}\leq m} |\rho^{\abs{\alpha}}D^\alpha g|_{\ell_2}^p} \tilde{\rho}^{\theta-2}\, dx\,dt <\infty,
\end{align*}
 for some $m\in\{0,1,2,\ldots\}$.
Then
\begin{equation}
    \label{eqn final}
\E\int_0^T \sum_{\abs{\alpha}\leq m+1}\int_{\cO} \Abs{\rho^{\abs{\alpha}-1}D^\alpha u}^p \tilde{\rho}^{\theta-2}\,dx\,dt\leq N\, C(m,\theta, f^i,f^0,g),
\end{equation}
where $N=N(p,\theta,\kappa_0,m, T)$.
\end{thm}

\begin{proof}
We prove the statement by induction over $m$. As in the proof of the results above, we use a partition of unity and apply corresponding results for the stochastic heat equation on angular domains (Corollary~\ref{high}) and on $\cont^1$ domains (\cite[Theorem~2.9]{Kim2004}) to estimate the solutions of the localized equations.

Let $r>0$, $\xi_j$, $\cD_j$, $j=1,\ldots,M$, as well as $\xi_0$ and $G$ be as in the proof of Lemma~\ref{lem for gronwall}.
In addition, assume that $G\subset\domain$ is chosen in such a way that
\[
\domain\setminus\bigcup_{j}B_{2r/3}(v_j) \subseteq G\subseteq \domain\setminus\bigcup_{j}B_{r/3}(v_j).
\]
As a consequence, 
\begin{equation}\label{eq:equiv:dist}
\rho_G\sim\rho_\cO\quad \text{and}\quad \tilde\rho\sim 1\quad \text{on}\quad\textup{supp}\,\xi_0\cap\domain.
\end{equation}
  
\smallskip
 
\noindent\emph{Step 1. The base case.} Let $m=0$. 
Note that in this case, the only difference in Estimate~\eqref{eqn final} compared to~\eqref{eqn polygon main} is the weight we put on $u$ on the left hand side of the inequality: $\rho^{-p}\tilde\rho^{\theta-2}$ in~\eqref{eqn final} instead of the smaller $\tilde\rho^{\theta-p-2}$ from~\eqref{eqn polygon main}.
But to obtain this sharper estimate we argue in a very similar fashion as in the proof of the latter with two changes: 
We use~Corollary~\ref{high} instead of Theorem~\ref{thm:ex:uni:2DCone} to estimate the solution in the vicinity of vertexes and we use the slightly modified choice of $G$ and~\eqref{eq:equiv:dist} to replace $\rho_G$ by $\rho_\domain$ after applying~\cite[Theorem~2.9]{Kim2004} to estimate the solution away from the vertexes. 
In detail, we argue as follows: The same reasoning as in the proof of Theorem~\ref{thm polygon main} shows that $\xi_0 u\in \frH^1_{p,\theta,0}(G,T)$ and $\xi_j u\in\cK^1_{p,\theta,0}(\cD_j,T)$ for $j\geq 1$ satisfy~\eqref{eqn 4.11.10} on $G$ and on $\cD_j$, $j\geq 1$, respectively. 
In particular, if $1\leq j\leq M$, then 
by Corollary~\ref{high} (see also Remark~\ref{remark cones}), Estimate~\eqref{eqn final} holds  with $\xi_j u$ and $C(0, \theta,f^{j,i}, f^{j,0},\xi_jg)$ in place of $u$ and $C(0,\theta,f^i,f^0, g)$, respectively. Here $f^{j,i}$ and $f^{j,0}$ are taken from \eqref{fji} and \eqref{fj0}.
Moreover, by the corresponding result on $\cont^1$ domains (see \cite[Theorem~2.9]{Kim2004}) and \eqref{eq:equiv:dist}, Estimate~\eqref{eqn final} also holds for $\xi_0u$ and  $C(0,\theta, f^{0,i}, f^{0,0},\xi_0g)$ in place of $u$ and $C(0,\theta,f^i,f^0, g)$, respectively.  
 Summing up all these estimates and using the second relationship in~\eqref{eq:equiv:dist} yields 
\begin{align*}
\E\int_0^T \sum_{\abs{\alpha}\leq 1} \int_{\cO} \Abs{\rho^{\abs{\alpha}-1}D^\alpha u}^p \tilde{\rho}^{\theta-2}\,dx\,dt
&\leq N \sum_{j=0}^{M} C(0,\theta, f^{j,i},f^{j,0},\xi_j g)\\
&\leq N \,\|u\|^p_{\bL^{[\circ]}_{p,\theta-p}(\cO)}+ N C(0,\theta,f^i,f^0,g) \\
&\leq N\, C(0,\theta,f^i,f^0,g);
\end{align*}
the last inequality above is due to \eqref{eqn polygon main}. The base case is proved.

\smallskip
 
\noindent\emph{Step 2. The induction step.} Suppose  that \eqref{eqn final} holds for some $m\in\{0,1,\ldots\}$ and $C(m+1,\theta,f^i,f^0,g)<\infty$. Then, by assumption,
\begin{equation}\label{eqn 4.13.2.a}
\E\int_0^T  \sum_{\abs{\alpha}\leq m+1}\int_{\cO} \Abs{\rho^{\abs{\alpha}-1}D^\alpha u}^p \tilde{\rho}^{\theta-2}\,dx\,dt\leq N C(m,\theta, f^i,f^0,g).
\end{equation}
Using \eqref{eqn 4.13.2.a}, one can easily check that
 $$
 \sum_{j=0}^{M} C(m+1,\theta,f^{j,i}, f^{j,0},\xi_j g) \leq N\, C(m+1,\theta,f^i,f^0,g).
 $$ 
Therefore, appropriate applications of Corollary~\ref{high} (see also Remark~\ref{remark cones}) and \cite[Theorem~2.9]{Kim2004} yield suitable estimates of $\sum_{|\alpha|\leq m+2} \E \int^T_0 |\rho^{|\alpha|-1}D^{\alpha}(\xi_j u)|^p \tilde{\rho}^{\theta-2}dxdt$ for $j=0$ and $1\leq j\leq M$,  respectively, which, summed up, yield
\begin{align*}
\E\int_0^T \sum_{\abs{\alpha}\leq m+2} \int_{\cO}& \Abs{\rho^{\abs{\alpha}-1}D^\alpha u}^p \tilde{\rho}^{\theta-2}\,dx\,dt\\
& \leq N \sum_{j=0}^{M}
 C(m+1,\theta,f^{j,i},f^{j,0},\xi_j g)  \leq N\, C(m+1,\theta,f^i,f^0,g). 
\end{align*}
Thus the induction goes through and the theorem is proved.
\end{proof}

\providecommand{\bysame}{\leavevmode\hbox to3em{\hrulefill}\thinspace}
\providecommand{\MR}{\relax\ifhmode\unskip\space\fi MR }
\providecommand{\MRhref}[2]{%
  \href{http://www.ams.org/mathscinet-getitem?mr=#1}{#2}
}
\providecommand{\href}[2]{#2}

\end{document}